\documentclass[a4paper, 10pt]{article}
\usepackage[latin1]{inputenc}
\usepackage[english]{babel}

\usepackage{amsmath}
\usepackage{amsthm}
\usepackage{amssymb}
\usepackage{color}
\usepackage{bbm}

\theoremstyle{plain}
\newtheorem{theorem}{Theorem}
\newtheorem{proposition}[theorem]{Proposition}
\newtheorem{lemma}[theorem]{Lemma}
\newtheorem{corollary}[theorem]{Corollary}
\newtheorem{remark}[theorem]{Remark}

\newcommand{\R}{\mathbb{R}}

\newcommand{\mD}{\mathcal{D}}
\newcommand{\mL}{\mathcal{L}}
\newcommand{\mM}{\mathcal{M}}
\newcommand{\mR}{\mathcal{R}}
\newcommand{\mT}{\mathcal{T}}
\newcommand{\mV}{\mathcal{V}}
\newcommand{\dd}{\, \text{d}}

\newcommand{\be}{\mathbf{e}}
\newcommand{\bff}{\mathbf{f}}
\newcommand{\bg}{\mathbf{g}}
\newcommand{\bp}{\mathbf{p}}

\newcommand{\bv}{\mathbf{v}}
\newcommand{\bw}{\mathbf{w}}
\newcommand{\by}{\mathbf{y}}
\newcommand{\bz}{\mathbf{z}}

\newcommand{\bbH}{\mathbb{H}}
\newcommand{\bbL}{\mathbb{L}}

\DeclareMathOperator{\divv}{div}

\DeclareMathOperator{\Symm}{Sym}

\newtheorem{innercustomthm}{Consequence}
\newenvironment{customthm}[1]
{\renewcommand\theinnercustomthm{#1}\innercustomthm}
{\endinnercustomthm}

\newtheorem{innercustomthmm}{Assumption}
\newenvironment{customthmm}[1]
{\renewcommand\theinnercustomthmm{#1}\innercustomthmm}
{\endinnercustomthmm}

\usepackage{hyperref}
\usepackage{vmargin}
\setmarginsrb{3cm}{2.5cm}{3cm}{1.5cm}{0cm}{0cm}{0cm}{1.5cm}

\begin{document}

\title{Feedback Stabilization of the Three-Dimensional Navier-Stokes Equations using Generalized Lyapunov Equations}
\author{Tobias Breiten\footnote{Institute of Mathematics, University of Graz, Austria. E-mail: tobias.breiten@uni-graz.at} \quad
Karl Kunisch\footnote{Institute of Mathematics, University of Graz, Austria and RICAM Institute, Austrian Academy of Sciences, Linz, Austria. E-mail: karl.kunisch@uni-graz.at}}

\maketitle

\begin{abstract}
The approximation of the value function associated to a stabilization problem formulated as optimal control problem for the Navier-Stokes equations in dimension three  by means of solutions to generalized Lyapunov equations is proposed and analyzed. The specificity, that the value function is not differentiable on the state space must be overcome. For this purpose a new class of generalized Lyapunov equations is introduced. Existence of unique solutions to these equations is demonstrated. They provide the basis for  feedback operators, which approximate the value function, the optimal states and controls, up to arbitrary order.
 \end{abstract}



\section{Introduction}

This work is concerned with feedback stabilization of the 3-D Navier-Stokes equations around a possibly  unstable stationary solution.
The approximation is achieved  by Taylor series-like expansions of the value function associated to an infinite-horizon optimal control problem. A related goal was achieved in  \cite{BreKP19b} for the two-dimensional case. But the approach from \cite{BreKP19b}
 cannot be generalized to the 3-D case since it builds on the differentiability of the value function on the state space. This  is not possible in dimension 3 and thus an independent approach and analysis is required.

Indeed the difficulty that arises is related to the lack of a convenient energy equality for the Navier-Stokes equations in dimension 3. Such an equality is available in dimension 2 and it is the basis for proving the uniqueness of weak  variational solutions of the Navier-Stokes equations with initial data in $\bbL^2(\Omega)$, where $\bbL^2(\Omega)$ denotes square integrable vector-valued  functions over $\Omega$. In dimension 3 we must resort to strong variational solutions with initial data in $\bbH^1(\Omega)$. As a consequence we can expect that the value function associated to optimal control problems of the Navier-Stokes equations in dimension 3 is well-defined and enjoys certain smoothness properties in $\bbH^1(\Omega)$ but not over $\bbL^2(\Omega)$. Having in mind that Taylor expansions to nonlinear operators on a space $X$ involve multilinear forms on product spaces consisting of copies of $X$, it becomes clear that $X= \bbH^1(\Omega)$ is not a convenient space to work with, especially if ultimately,
numerical realizations are desired. For this purpose  $X= \bbL^2(\Omega)$ is significantly more convenient. Here we aim for an approximation of the value function with operators constructed in an  $\bbL^2(\Omega)$-setting, in spite of the fact that the value function is differentiable on $\bbH^1(\Omega)$ only. These operators will be constructed as the solutions to generalized Lyapunov equations.

We next introduce the specific problem of interest. Let $\Omega \subset \mathbb R^3$ denote a bounded domain with $C^{1,1}$ boundary $\Gamma$, and let $\tilde{B}$ denote a control to state operator. We aim at designing a control $u$ such that the solution $(\bz,q)$ to the time-dependent Navier-Stokes equations
\begin{equation}\label{eq:inhom_NSE}
  \begin{aligned}
   \frac{\partial \bz}{\partial t} &= \nu \Delta \bz - (\bz \cdot \nabla)\bz - \nabla q +\boldsymbol{\varphi} + \tilde{B} u &&  \text{in } \Omega \times (0,T), \\
   \divv \bz&= 0 &&  \text{in } \Omega \times (0,T), \\
   \bz &= \boldsymbol{\psi} &&  \text{on } \Gamma \times (0,T), \\
   \bz(0) &= \bar{\bz} + \by_0 &&
  \end{aligned}
\end{equation}
satisfies $\lim\limits_{t\to \infty} \bz(t) = \bar{\bz}$ for perturbations $\by_0$ with $\divv\by_0=0$, which are assumed to be suitably \emph{small}. Here $\bar \bz$  is the velocity component of the solution $(\bar{\bz},\bar{q})$  to the  stationary Navier-Stokes equations
\begin{equation}\label{eq:stat_NSE}
 \begin{aligned}
  -\nu \Delta \bar{\bz} + (\bar{\bz} \cdot \nabla ) \bar{\bz} + \nabla \bar{q}  &= \boldsymbol{\varphi} && \text{in } \Omega, \\
  \divv\bar{\bz} &= 0 && \text{in } \Omega, \\
  \bar{\bz} &=\boldsymbol{\psi} && \text{on } \Gamma,
 \end{aligned}
\end{equation}
for given vector-valued functions $\boldsymbol{\varphi}$ and $\boldsymbol{\psi}$. All regularity assumptions will be specified below.

With the intention of formulating this problem as an infinite-horizon control problem, we define $(\by,p):=(\bz,q)-(\bar{\bz},\bar{q})$ and, instead of \eqref{eq:inhom_NSE}, we focus on the following generalized Navier-Stokes equations
\begin{equation}\label{eq:gen_hom_NSE}
  \begin{aligned}
   \frac{\partial \by}{\partial t} &= \nu \Delta \by - (\by \cdot \nabla)\bar{\bz} - (\bar{\bz} \cdot \nabla)\by-(\by \cdot \nabla)\by - \nabla p + \tilde{B} u &&  \text{in } \Omega \times (0,T), \\
   \divv\by&= 0 &&  \text{in } \Omega \times (0,T), \\
   \by &= 0 &&  \text{on } \Gamma \times (0,T). \\
   \by(0) &= \by_0&&
  \end{aligned}
\end{equation}
Our goal consists in proving  that $\lim\limits_{t\to \infty} (\by(t),p(t))=(0,0)$. To achieve this, we focus on the following problem:
\begin{equation} \label{eq:NLQprob_intro} \tag{$\widetilde{P}$}
\inf_{\begin{subarray}{c} \by \in W_\infty(\mD(A_\lambda),Y) \\ u \in L^2(0,\infty;U) \end{subarray}} \frac{1}{2} \int_0^\infty \| \by \|^2_Y \dd t + \frac{\alpha}{2} \int_0^\infty \| u(t) \|_U^2 \dd t, \quad \text{subject to: } \eqref{eq:gen_hom_NSE},
\end{equation}
where the spaces $\mD(A_\lambda), Y$ and $U$ are classical function spaces related to the Navier-Stokes equations that will be introduced below.

Let us mention some references that address problems similar to the one considered here. Regarding feedback control of the (three-dimensional) Navier-Stokes equations, we point to, e.g., \cite{Bad12,BadT11,BarLT06,BarRS11,BarT04,Fur04,Ray07} where different feedback methodologies based on spectral decomposition or Riccati equations have been analyzed. For (local) exact null controllability results of the (linearized) Navier-Stokes equations, see, e.g., \cite{FGIP2011,Ima01}. The idea of approximating the optimal feedback law by utilizing a Taylor series expansion of the minimal value function has its origin in finite-dimensional considerations proposed in \cite{Alb61,Luk69} which, later on, have been picked up in, e.g., \cite{AguK14,NavK07}. For a survey summarizing the approach (and related variants), we refer to \cite{BeeTB00}. One of the first references dealing with polynomial feedback laws of infinite-dimensional control systems can be found in \cite{TheBR10}. For the special class of bilinear control systems, Taylor series expansions have recently been analyzed in detail in \cite{BreKP18,BreKP19}. 
The literature on open loop control of the Navier Stokes equation is quite rich. Topics such as necessary and sufficient optimality conditions, numerical approximation of the optimality systems are well investigated, in general. If one focuses on the work dedicated to time dependent optimal control in three space dimensions the literature is scars, however. Here we mention \cite{PR07,TW06}, for  finite horizon optimal control problems subject to the Navier-Stokes equations. 

The contents of the paper is structured as follows. Section 2 contains the problem statement and function space preliminaries. Differentiability properties of the value function are discussed in Section 3. The subsequent section is devoted to the introduction and analysis of the generalized Lyapunov equations. This leads to the polynomial feedback laws which are described in Section 5. Section 6 contains the error estimates for the value function, the optimal states, and controls. We finish with short conclusions.

\paragraph{Notation.} For Hilbert spaces $V\subset Y$ with dense and compact embedding, we consider the Gelfand triple $V\subset Y \subset  V'$ where $V'$ denotes the topological dual of $V$ with respect to the pivot space $Y$. For vector-valued functions $\bff \in (L^2(\Omega))^3$, we use the notation $\bff\in \mathbb{L}^2(\Omega)$. Elements $\bff \in \mathbb{L}^{2}(\Omega)$ will be denoted in boldface and distinguished from scalar-valued functions $g \in L^2(\Omega)$. Similarly, we use $\mathbb{H}^2(\Omega)$ for the space $(H^2(\Omega))^3$. For a closed, densely defined linear operator $(A,\mD(A))$ in $Y$, its adjoint (again considered as an operator $Y$) will be denoted with $(A^*,\mD(A^*))$. Considering $A$ as a bounded linear operator $A\in \mL(\mD(A),Y)$ its dual $A'\in \mL(Y,[\mD(A)]')$ is uniquely defined. Let us recall that it is the unique extension of the operator $A^* \in \mL(\mD(A^*),Y)$ to an element of $\mL(Y,[\mD(A)]')$.
In fact, we have
\begin{equation*}
\langle A^*p,y \rangle_Y= \langle p,A y \rangle_Y \qquad \text{ for all } p\in \mD(A^*), \text{ and }  y\in \mD(A),
\end{equation*}
 and
\begin{equation*}
\langle A y, p \rangle_{Y} = \langle y, A' p \rangle_{\mD(A),[\mD(A)]'}  \qquad \text{ for all } p\in Y, \text{ and } y\in \mD(A).
\end{equation*}
Since $\mD(A^*)$ is dense  in $Y$, this implies that $A'$ is the unique extension of $A^*$ to an operator in $\mathcal{L}(Y,[\mD(A)]')$. For an infinitesimal generator $A$ of an exponentially stable semigroup  $e^{At}$ on $Y$, we consider the space $W(0,T;\mD(A),Y)$ which we endow with the norm
\begin{align}\label{eq:Wnorm}
\| y \| _{ W(0,T;\mD(A),Y)} := \left( \| Ay \|_{L^2(0,T;Y)}^2 + \|\frac{\mathrm{d}}{\mathrm{d}t}y \| _{L^2(0,T;Y)}^2 \right)^\frac{1}{2} , \ \ y \in W(0,T;\mD(A),Y).
\end{align}
Generally, given $T\in \mathbb R$ and two Hilbert spaces $X\subset Y$, by $W(0,T;X,Y)$ we denote the space
\begin{align*}
W(0,T;X,Y) =\left\{ y \in L^2(0,T;X) \ | \ \frac{\mathrm{d}}{\mathrm{d}t} \in L ^2(0,T;Y) \right\}.
\end{align*}
For $T=\infty$, the space $W(0,T;X,Y)$ will be denoted by $W_\infty(X,Y)$.  For $\delta \geq 0$, we denote by $B_Y(\delta)$ the closed ball in $Y$ with radius $\delta$ and center 0.

For $k \geq 1$, we make use of the following norm:
\begin{equation} \label{eq:NormYk}
\| (v_1,\dots,v_k) \|_{V^k}= \max_{i=1,\dots,k} \| v_i \|_V,
\end{equation}
on the product space $V^k:=V\times \cdots \times V$.
Given a Hilbert space $Z$, we say that $\mathcal{T}\colon V^k \rightarrow Z$ is a bounded multilinear mapping (or bounded multilinear form for $Z= \R$) if for all $i \in \{ 1,\dots,k \}$ and for all $(v_1,\dots,v_{i-1},v_{i+1},\dots,v_k) \in V^{k-1}$, the mapping
$v\in V \mapsto \mathcal{T}(v_1,\dots,v_{i-1},v,v_{i+1},\dots,v_k) \in Z$ is linear and
\begin{equation} \label{eq:OperatorNormTensor}
\| \mathcal{T} \| := \sup_{v \in B_{V^k}(1)} \| \mathcal{T}(v) \|_Z < \infty.
\end{equation}
The set of bounded multilinear mappings on $V^k$ will be denoted by $\mathcal{M}(V^k,Z)$. For all $\mathcal{T} \in \mathcal{M}(V^k,Z)$ and for all $(v_1,\dots,v_k) \in V^k$, we have
\begin{equation*}
\| \mathcal{T}(v_1,\dots,v_k) \|_Z \leq \| \mathcal{T} \| \, \prod_{i=1}^k \| v_i \|_V.
\end{equation*}
Bounded multilinear mappings $\mathcal{T} \in \mathcal{M}(V^k,Z)$ are said to be symmetric if for all $v_1,\dots,v_k \in V^k$ and for all permutations $\sigma$ of $\{ 1,\dots,k \}$,
\begin{equation*}
\mathcal{T}(v_{\sigma(1)},\dots,v_{\sigma(k)})= \mathcal{T}(v_1,\dots,v_k).
\end{equation*}
Finally, given two multilinear mappings $\mathcal{T}_1 \in \mathcal{M}(V^k,Z)$ and $\mathcal{T}_2 \in \mathcal{M}(V^{\ell},Z)$, we denote by $\mathcal{T}_1 \otimes \mathcal{T}_2 \in \mathcal{M}(V^{k+\ell},Z)$ the bounded multilinear form defined by
\begin{equation*}
\mathcal{T}_1 \otimes \mathcal{T}_2(v_1,\dots,v_{k+ \ell})
= \langle \mathcal{T}_1(v_1,\dots,v_k), \mathcal{T}_2(v_{k+1},\dots,v_{k+\ell}) \rangle_Z.
\end{equation*}

Throughout the manuscript, we use $M$ as a generic constant that might change its value between consecutive lines.

\section{Analytical preliminaries}

\subsection{Function spaces}

Let us briefly summarize the classical functional analytic framework that allows us to consider \eqref{eq:gen_hom_NSE} as an abstract differential equation on the space of solenoidal vector fields. Based on this formulation, we subsequently define our stabilization problem of interest.
For more details on the following well-known decomposition, let us refer to, e.g., \cite{Bar11,BarLT06,Fur01,Ray06,Tem79} for details. We introduce the spaces
\begin{align*}
  Y&:=\left\{ \by \in \bbL^2(\Omega)\, |\, \divv\by=0, \by\cdot \vec{n}=0 \text{ on } \Gamma \right\}, \\
  V&:=\left\{ \by \in \bbH_0^1(\Omega)\, | \, \divv\by=0 \right\},
\end{align*}
endowed with the canonical inner products and norms. Note that
$Y$ is a closed subspace of $\bbL^2(\Omega)$ which is associated to the orthogonal decomposition
\begin{align}\label{eq:orth_dec_div_free}
  \bbL^2(\Omega) = Y \oplus Y^\perp,
\end{align}
where
\begin{align}\label{eq:Yperp}
  Y^\perp = \left\{\bz =\nabla p \, | \, p \in H^1(\Omega) \right\}.
\end{align}
In this context, we recall the  \textit{Leray projector} $P\colon \bbL^2(\Omega) \to Y$ which orthogonally projects $\bbL^{2}(\Omega)$ onto $Y$.
Consider the  nonlinear operator $F\colon \mathbb H^2(\Omega)\cap V\to Y$ defined by
\begin{align}\label{eq:nonlin}
 F(\by) = P((\by\cdot \nabla)\by).
\end{align}
Let us further define the bilinear mapping $N(\by,\bz):=P((\by\cdot \nabla)\bz)$ for which we recall the following properties:

\begin{proposition}\label{prop:N_estimates}
Let $\Omega$ be bounded domain of class $C^{1,1}$ in $\mathbb{R}^3$. Then there exists a constant $M$ such that
\begin{itemize}
  \item[(i)] $\|N(\by,\bz)\|_{Y}\le M \|\by\|_{\bbH^2(\Omega)} \| \bz\|_{V} $, \text{for all} $\by \in \bbH^2(\Omega) \cap V, \bz \in V$,
  \item[(ii)]  $\|N(\by,\bz)\|_{Y}\le M \|\by\|_{V} \| \bz\|_{\bbH^2(\Omega)} $, \text{for all} $\by \in V, \bz \in \bbH^2(\Omega)\cap V$,
  \item[(iii)] $\|N(\by,\bz)\|_{V}\le M \|\by\|_{\bbH^2(\Omega)} \| \bz\|_{\bbH^2(\Omega)} $, \text{for all} $\by,\bz \in \bbH^2(\Omega) \cap V$.
\end{itemize}
\end{proposition}
\begin{proof} The first two properties follow from the standard Sobolev embedding results $H^2(\Omega)\hookrightarrow C(\bar{\Omega}),\, H^1(\Omega) \hookrightarrow L^4(\Omega),$ and $H^2(\Omega)\hookrightarrow W^{1,4}(\Omega)$. For the third one,  in addition we use that $P\in {\mathcal L}({\mathcal\bbH^1(\Omega)})$, \cite[Proposition 4.3.7]{Boy13}. Here the $C^{1,1}$ property of the domain is used.
%
\end{proof}
We shall also consider $N$ as a bilinear mapping from $V\times V$ to $V'$, which is defined by
\begin{align}\label{eq:bilform_weak}
  N\colon   V\times  V \to V', \ \  \langle N(\by,\bz),\bw \rangle_{V',V}:=\langle (\by\cdot \nabla)\bz, \bw \rangle_{V',V} .
\end{align}
We have the following properties, which can again be verified by standard Sobolev embedding results, and the fact that
$\langle (\by\cdot \nabla)\bz, \bw \rangle_{V',V}=\langle (\by\cdot \nabla)\bw, \bz \rangle_{V',V}$, for all $(\by,\bz, \bw)\in V^3$.

\begin{proposition}\label{prop:N_estimates_2}
Let $\Omega$ be a bounded Lipschitz domain in $\mathbb R^3$. Then there exists a constant $M$ such that
\begin{itemize}
  \item[(i)] $\|N(\by,\bz)\|_{V'}\le M \|\by\|_{V} \| \bz\|_{V} $, \text{for all} $\by, \bz \in V$,
  \item[(ii)] $\|N(\by,\bz)\|_{V'}\le M \|\by\|_{\bbH^{2}(\Omega)} \| \bz\|_{Y} $, \text{for all} $\by \in \bbH^2(\Omega)\cap V, \bz \in Y$,
  \item[(iii)] $\|N(\by,\bz)\|_{V'}\le M \|\by\|_{Y} \| \bz\|_{\bbH^{2}(\Omega)} $, \text{for all} $\by\in Y, \bz \in \bbH^2(\Omega)\cap V$.
  \end{itemize}
\end{proposition}
Analogous properties can be obtained for the nonlinear operator $F$.


The \emph{Oseen-Operator} is defined by
\begin{align}\label{eq:oseen}
	A_{0} \colon (\mathbb H^2(\Omega)\cap V)\times (\mathbb H^2(\Omega)\cap V) \to Y, \ \ A_{0}(\by,\bz):= N(\by,\bz)+N(\bz,\by).
\end{align}
Given a stationary solution $\bar{\bz} \in V$, we associate with it the \textit{Stokes-Oseen} operator $A$ that is defined as follows
\begin{equation}\label{eq:A_NSE}
  \mD(A) = \bbH^2(\Omega)\cap V, \ \ A\by = P(\nu \Delta \by - (\by \cdot \nabla)\bar{\bz} - (\bar{\bz} \cdot \nabla)\by ).
\end{equation}
Considered as operator in $\bbL^2(\Omega)$ the adjoint $A^*$, again as operator in $\bbL^2(\Omega)$, can be characterized by
\begin{equation}\label{eq:Aadj_NSE}
  \mD(A^*) = \bbH^2(\Omega)\cap V, \ \ A^*\bp = P(\nu \Delta \bp - (\nabla \bar{\bz})^T \bp + (\bar{\bz} \cdot \nabla)\bp ).
\end{equation}
As mentioned before, considering $A$ as an element of $\mathcal{L}(\mD(A),Y)$ its dual $A'\in \mathcal {L}(Y,[\mD(A)]')$ is the unique extension of  the operator $A^* \in \mathcal{L}(\mD(A),Y)$ to an element in $\mathcal{L} (Y,[\mD(A)]')$.

For the  control operator $\tilde{B}$ we assume that $\tilde B\in \mathcal{L}(U,\bbL^2(\Omega))$. Let us set $B:=P\tilde{B} \in \mathcal{L}(U,Y)$. We are now prepared to project the controlled state equation \eqref{eq:gen_hom_NSE}  onto the space of solenoidal vector fields $Y$. We arrive at the abstract control system
\begin{equation}\label{eq:abs_Cauchy}
\begin{aligned}
  \frac{\dd }{\dd t}\by(t) &= A\by - F(\by) + Bu, \ \  \by(0) =\by_0,
\end{aligned}
\end{equation}
where the pressure $p$ is eliminated. Before we state the optimal control problem, we collect some generalizations of the estimates in Proposition \ref{prop:N_estimates} to the time-varying case.

\begin{corollary}\label{cor:N_time_estimates}
Let $\by,\bz \in L^2(0,\infty;\mD(A)),\bv \in L^\infty(0,\infty;V)$. Then
  \begin{align}
   \| N(\by,\bz) \|_{L^1(0,\infty;V)} &\le M \| \by \|_{L^2(0,\infty;\bbH^2(\Omega))} \| \bz \|_{L^2(0,\infty;\bbH^2(\Omega))} , \ \ \label{eq:N_L1} \\
   \| N(\by,\bz) \|_{L^1(0,\infty;V')}& \le M \| \by \|_{L^2(0,\infty;\bbH^2(\Omega))} \| \bz \|_{L^2(0,\infty;Y)} , \ \ \label{eq:N_L1b} \\
   \| N(\by,\bz) \|_{L^1(0,\infty;V')}& \le M \| \by \|_{L^2(0,\infty;Y)} \| \bz \|_{L^2(0,\infty;\bbH^2(\Omega))} , \ \ \label{eq:N_L1c} \\
  \| N(\by,\bv) \|_{L^2(0,\infty;Y)} &\le M \| \by \|_{L^2(0,\infty;\bbH^2(\Omega))} \| \bv \|_{L^\infty(0,\infty;V)}, \ \ \label{eq:N_L2a}   \\
  \| N(\bv,\bz) \|_{L^2(0,\infty;Y)} &\le M  \| \bv \|_{L^\infty(0,\infty;V)} \| \bz \|_{L^2(0,\infty;\bbH^2(\Omega))}. \ \ \label{eq:N_L2b}
  \end{align}
\end{corollary}

\begin{corollary}\label{cor:F_Lip}
 For all $\by,\bz \in L^2(0,\infty;\mD(A)) \cap L^\infty(0,\infty;V)$ with $\max(\|\by \|_{L^\infty(0,\infty;V)}, \| \bz \| _{L^\infty(0,\infty;V)} )\le \delta$, there exists a constant $C>0$ such that
  \begin{align}\label{eq:F_Lip}
   \|F(\by)-F(\bz)\|_{L^2(0,\infty;Y)} \le \delta C\| \by-\bz\|_{L^2(0,\infty;\bbH^2(\Omega))}.
 \end{align}
\end{corollary}

\begin{proof}
  Note that
 \begin{align*}
 \| F(\by)-F(\bz)\|_{L^2(0,\infty;Y)} &= \| P(N(\by,\by)-N(\bz,\bz))\|_{L^2(0,\infty;Y)} \\
 &\le C (\|N(\by-\bz,\by)\|_{L^2(0,\infty;Y)}+\| N(\bz,\by-\bz)\|_{L^2(0,\infty;Y)} )
 \end{align*}
 which, together with \eqref{eq:N_L2a} and \eqref{eq:N_L2b}, shows the assertion.
\end{proof}

\subsection{Existence of feasible solutions}

Throughout the article we assume that the following assumptions hold true.
\begin{customthmm}{A1}\label{ass:A1}
 The stationary solution satisfies $\bar{\bz} \in V$.
 \end{customthmm}
 \begin{customthmm}{A2}\label{ass:A2}
 The linearized system $(A,B)$ is exponentially stabilizable, i.e., there exists $K \in \mathcal{L}(Y,U)$ such that the semigroup $e^{(A-BK) t}$ is exponentially stable on $Y$.
\end{customthmm}
Regarding Assumption \ref{ass:A2}, we refer to, e.g.,  \cite{Bar11} where finite-dimensional
feedback operators are constructed on the basis of spectral decomposition as well as Riccati theory.
Alternatively, exponential stabilizability of the linearized system also follows from  exact controllability results available  in
\cite{FGIP2011}.

We immediately obtain the following important consequences that will be used several times throughout the manuscript.

\begin{customthm}{C1}\label{cons:C1} There exist two constants $\lambda \ge 0$ and $\theta >0$ such that
 \begin{align}\label{eq:VY-coerc}
 \langle \underbrace{(\lambda I-A)}_{:=-A_\lambda}\bv,\bv \rangle_Y \ge \theta \|\bv \|_V^2 \text{ for all } \bv \in V
\end{align}
As is well-known, see e.g., \cite[Theorem II.1.2.12]{Benetal07} estimate \eqref{eq:VY-coerc} implies that $A_\lambda$ and, thus, $A$ generate
analytic semigroups $e^{A_\lambda t}$ and $e^{At}$, on $Y$, respectively.
\end{customthm}

Below we will frequently make use of the spaces $W(0,T;\mD(A_\lambda),Y)$ and $W_\infty(\mD(A_\lambda),Y)$, respectively, endowed with the norm defined in \eqref{eq:Wnorm}. This notation will also be employed for systems associated to operators $A$ that do not necessarily generate themselves  an exponentially stable semigroup on $Y$.

\begin{customthm}{C2}\label{cons:C2}
  For all $(\by_0,\bff) \in V \times \in L^2(0,\infty;Y)$ and $T>0$ the system
\begin{align}\label{eq:auxC2}
 \dot{\by}=A\by + \bff , \quad \by(0)=\by_0,
\end{align}
has a unique solution  $\by \in W(0,T;\mD(A_\lambda),Y)$. In addition, this solution  satisfies
\begin{align*}
 \| \by \|_{W(0,T;\mD(A_\lambda),Y)} \le c(T) (\| \by _0\|_V + \| \bff \| _{L^2(0,\infty;Y)} )
\end{align*}
with a continuous function $c$. In the case  that $\by \in L^2(0,\infty;Y)$, we can replace \eqref{eq:auxC2} by the equivalent formulation
\begin{align*}
 \dot{\by}=\underbrace{(A- \lambda I)}_{A_\lambda} \by + \underbrace{\lambda \by + \bff }_{\bff_\lambda}, \quad \by (0)=\by_0,
\end{align*}
with $\bff_\lambda\in L^2(0,\infty;Y)$. Since $A_\lambda$ satisfies \eqref{eq:VY-coerc}, from, e.g., \cite[Part II, Chapter 1, Section 2]{Benetal07} it follows that it generates an analytic, exponentially stable semigroup on $Y$ and thus by, e.g.,  \cite[Theorem II.1.3.1]{Benetal07} there exists $M_\lambda$ such that $\by \in W_\infty(\mD(A_\lambda),Y)$ with
\begin{align}\label{eq:reg_est_shift}
  \| \by \| _{W_\infty(\mD(A_\lambda),Y)} \le M_{\lambda} ( \| \by _0 \| _V + \| \bff_{\lambda} \| _{L^2(0,\infty;Y)} ).
\end{align}
Here, we use that  $[\mD(A_\lambda),Y]_{\frac{1}{2}}=V,$ see \cite[Theorem II.2.1.3]{Benetal07}.
\end{customthm}
\begin{customthm}{C3}\label{cons:C3}  For all $\by_0\in V$ and $\bff \in L^2(0,\infty;Y)$, the system
\begin{align*}
 \dot{\by}=(A-BK)\by + \bff, \quad \by(0)=\by_0,
\end{align*}
has a unique solution in $W_\infty(\mD(A_\lambda),Y)$, see \cite[Theorem II.1.3.1]{Benetal07}. In particular, it holds that
\begin{align}\label{eq:reg_est_lin}
  \| \by\|_{W_\infty(\mD(A_\lambda),Y)} \le M_K (\| \by_0 \|_V + \| \bff \|_{L^2(0,\infty;Y)} ).
\end{align}
 \end{customthm}

In the next lemma, by $A_s$ we denote an abstract generator of an exponentially stable, analytic semigroup on $Y$. The proof of the assertion is based on a classical fixed-point argument which has been used in similar contexts in, e.g., \cite{BreKP19b,Ray06} and is given in the Appendix \ref{sec:appendix_proofs}.

\begin{lemma} \label{lem:non_loc_sol}
  Let $A_s$ generate an exponentially stable, analytic semigroup $e^{A_s t}$ on $Y$, let $C$ denote the constant specified in  Corollary \ref{cor:F_Lip} and  let $F$ be as in \eqref{eq:nonlin}. Then there exists a constant $M_s$ such that for all $(\by_0,\bff) \in  V\times L^2(0,\infty;Y)$ with
  \begin{align*}
\gamma:=    \| \by _0\| _V + \| \bff \| _{L^2(0,\infty;Y)} \le \frac{1}{4CM_s^2}
  \end{align*}
  the system
  \begin{equation}\label{eq:non_loc_sol}
\dot{\by}= A_s\by - F(\by) +\bff, \quad \by(0)= \by_0,
\end{equation}
has a unique solution $\by \in W_\infty(\mD(A_s),Y)$. Moreover, we have the following estimate on $\by$:
\begin{align*}
\| \by \|_{W_\infty(\mD(A_s),Y)}\le  2 M_s \gamma.
\end{align*}
\end{lemma}

\begin{corollary}\label{cor:feas_control}
There exists a constant $M_K>0$ such that for all $(\by_0,\bff) \in V\times \in L^2(0,\infty;Y)$ which satisfy
\begin{align*}
 \gamma:=\| \by_0 \|_V + \| \bff \| _{L^2(0,\infty;Y)} \le \frac{1}{4CM_K^2} 
\end{align*}
there exists $u \in L^2(0,\infty;U)$ such that there exists a unique solution $\by \in W_\infty(\mD(A_\lambda),Y)$ to
\begin{equation}\label{eq:loc_stab_non2}
\dot{\by}= A\by + Bu - F(\by) +\bff, \quad \by(0)= \by_0.
\end{equation}
Additionally, it holds that
\begin{align*}
  \|\by\|_{W_\infty(\mD(A_\lambda),Y)} \le 2M_K \gamma  \quad \text{and} \quad \|u\|_{L^2(0,\infty;U)}\le 2 \|K\|_{\mathcal{L}(Y)} M_K \gamma.
\end{align*}
\end{corollary}
\begin{proof}
Since assumption \ref{ass:A2} implies the existence of $K$ such that $e^{(A-BK)t}$ is an exponentially stable analytic semigroup on $Y$, the result is a consequence of Lemma \ref{lem:non_loc_sol} applied to the system
    \begin{align*}
   \dot{\by}=(A-BK)\by -F(\by)+\bff , \quad \by(0)=\by_0.
  \end{align*}
  The estimate on the control follows from the feedback representation $u=-K\by$.
\end{proof}

For the next statement, we can w.l.o.g.\@ assume that $M_\lambda$ defined in Consequence \ref{cons:C2} satisfies  $M_\lambda \ge \frac{1}{2C}$.

\begin{corollary}\label{cor:more_regularity}
  Let $(\by_0,\bff)\in V \times L^2(0,\infty;Y)$ and $u \in L^2(0,\infty;U)$ be such that there exists a solution $\by \in L^2(0,\infty;Y)$ to
  \begin{align*}
    \dot{\by} = A\by - F(\by) + Bu + \bff, \ \ \by(0)=\by_0.
  \end{align*}
If $(\by_0,\bff,u,\by)$ are such that
  \begin{align*}
\gamma :=\| \by_0 \|_V + \| \bff + \lambda \by + Bu \| _{L^2(0,\infty;Y)} \le \frac{1}{4CM_\lambda^2} .
\end{align*}
then $\by \in W_\infty(\mD(A_\lambda),Y)$ and, moreover,
\begin{align*}
  \| \by \| _{W_\infty(\mD(A_\lambda),Y) } \le 2M_{\lambda} \gamma.
  \end{align*}
\end{corollary}
\begin{proof}
  By assumption it holds that $\by \in L^2(0,\infty;Y)$. As a consequence, we can focus on the equivalent  system
  \begin{align*}
   \dot{\by} = A_\lambda\by- F(\by) +\tilde{\bff},
  \end{align*}
where $\tilde{\bff} = \bff + \lambda \by + Bu$. Application of Lemma \ref{lem:non_loc_sol} then shows the assertion.
\end{proof}

With the previous considerations, we can state problem \eqref{eq:NLQprob_intro} as the following abstract infinite-horizon optimal control problem:
\begin{equation} \label{eq:NLQprob} \tag{$P$}
\inf_{\begin{subarray}{c} \by \in W_\infty(\mD(A_\lambda),Y) \\ u \in L^2(0,\infty;U) \end{subarray}} J(\by,u), \quad \text{subject to: } e(\by,u)= (0,\by_0)
\end{equation}
where $J \colon W_\infty(\mD(A_\lambda),Y) \times L^2(0,\infty;U) \rightarrow \R$ and $e \colon W_\infty(\mD(A_\lambda),Y) \times L^2(0,\infty;U) \rightarrow L^2(0,\infty;Y) \times V$ are defined by
\begin{align}
J(\by,u)= & \frac{1}{2} \int_0^\infty \| \by \|^2_Y \dd t + \frac{\alpha}{2} \int_0^\infty \| u(t) \|_U^2 \dd t \\
e(\by,u)= & \big( \dot{\by}-(A\by - F(\by) + Bu), \by(0) \big).
\end{align}

\section{Differentiability of the value function on $V$}

In this section, we show the differentiability on $V$ of the associated \emph{value function}, defined by
\begin{align*}
\mV(\by_{0}) = \inf_{\begin{subarray}{c} \by \in W_\infty(\mD(A_\lambda),Y) \\ u \in L^2(0,\infty;U) \end{subarray}} J(\by,u), \quad \text{subject to: } e(\by,u)= (0,\by_0).
\end{align*}
Our arguments are based on an analysis of the dependence of solutions to \eqref{eq:NLQprob} with respect to the initial condition $\by_{0}$.


\subsection{Existence of a solution and optimality conditions}

This section is devoted to existence of solutions to \eqref{eq:NLQprob} with small initial data and associated first-order necessary optimality conditions.

\begin{lemma}\label{lem:ex_opt_sol}
There exists $\delta_{1} >0 $  such that for all $\by_0 \in B_V(\delta_1)$
problem \eqref{eq:NLQprob} possesses a solution $(\bar{\by},\bar{u})$.
Moreover, there exists a constant $M> 0$ independent of $\by_0$ such that
  \begin{align}\label{eq:sol_est_NLQ}
   \max( \| \bar{u} \| _{L^2(0,\infty;U)},\| \bar{\by}\|_{W_\infty(\mathcal{D}(A_\lambda),Y)}) \le M \|\by_0\|_V.
  \end{align}
\end{lemma}
\begin{proof}
For now, let us define $\delta_{1} =\frac{1}{4CM_K^2}$ with $C$ as in Corollary \ref{cor:F_Lip} and $M_K$ as in Corollary \ref{cor:feas_control}. By Corollary \ref{cor:feas_control} (with $\bff=0$) there exists a control $u\in L^2(0,\infty;U)$ with associated state $\by$ satisfying
  \begin{align*}
   \max(  \| u \|_{L^2(0,\infty;U)},\|\by \|_{W_\infty(\mathcal{D}(A_\lambda),Y)} ) \le M \| \by_0 \|_V,
  \end{align*}
  where $M=2M_K \max(1,\| K\|_{\mathcal{L}(Y)})$.
  Let us now take a minimizing sequence $(\by_n,u_n)_{n\in\mathbb N}$ which we can assume to satisfy  $J(\by_n,u_n)\le M^2\|\by_0\|_V^2(1+\alpha)$. Consequently, for all $n \in \mathbb N$ we obtain
  \begin{align}\label{eq:aux_bound_opt}
  \| \by_n \|_{L^2(0,\infty; Y )} \le M \|\by_0\|_V \sqrt{2(1+\alpha)} \quad \text{and} \quad  \|u_n\| _{L^2(0,\infty;U)} \le M \|\by_0\|_V \frac{\sqrt{2(1+\alpha)}}{\sqrt{\alpha}}.
  \end{align}
   After possible reduction of $\delta_1$, we can assume that
  \begin{align*}
   \| \by_0\|_V + \| \lambda \by_n +Bu_n\|_{L^2(0,\infty;Y)} \le \left[1+M\sqrt{2(1+\alpha)} \left(\lambda + \frac{\|B\|_{\mathcal{L}(U,Y)}}{\sqrt{\alpha}} \right) \right]\delta_1 \le  \frac{1}{4CM_\lambda^2},
  \end{align*}
where $M_\lambda$ is as in Consequence \ref{cons:C2} and Corollary \ref{cor:more_regularity}. Hence, the sequence $(\by_n)_{n \in \mathbb N}$ is bounded in $W_\infty(\mD(A_\lambda),Y)$ with $\sup\limits_{n\in \mathbb N}\| \by_n\| \le 2M_\lambda \| \by_0 \|_V$. Extracting if necessary a subsequence, there exists $(\bar{\by},\bar{u})\in W_\infty(\mD(A_\lambda),Y) \times L^2(0,\infty;U)$ such that $(\by_n,u_n) \rightharpoonup (\bar{\by},\bar{u}) \in W_\infty(\mD(A_\lambda,Y) \times L^2(0,\infty;U)$, and $(\bar{\by},\bar{u})$ satisfies \eqref{eq:sol_est_NLQ}.

We are going to prove that $(\bar{\by},\bar{u})$ is feasible and optimal.
Note that for each $T>0$ and an arbitrary $\bz \in H^{1}(0,T;Y)$, we have for all $n \in \mathbb N$ that
\begin{align}\label{eq:lim_state_eq}
 \int_0^T \left\langle \frac{\mathrm{d}}{\mathrm{d}t} \by_n(t),\bz(t) \right\rangle _{Y}\mathrm{d}t= \int_0^T \langle A\by _n (t) - F(\by_n(t))+B u_n(t),\bz(t) \rangle _{Y}\, \mathrm{d}t.
\end{align}
From the convergence $\frac{\mathrm{d}}{\mathrm{d}t}\by_n \rightharpoonup \frac{\mathrm{d}}{\mathrm{d}t} \bar{\by}$ in $L^2(0,T;Y)$, we can pass to to the limit in the l.h.s.\@ of the above equality. Similarly, using that $A\by_n \rightharpoonup A\bar{\by} \in L^2(0,T;Y)$  it follows that
\begin{align*}
 \int_0^T \langle A\by _n (t), \bz(t) \rangle _{Y}\, \mathrm{d}t \underset{n \to
\infty}{\longrightarrow}\int_0^T \langle A \bar{\by}(t),\bz(t) \rangle_{Y}\, \mathrm{d}t.
\end{align*}
With the same argument, we find that
\begin{align*}
 \int_0^T \langle Bu _n (t), \bz(t) \rangle _{Y}\, \mathrm{d}t \underset{n \to
\infty}{\longrightarrow}\int_0^T \langle B\bar{u}(t),\bz(t) \rangle_{Y}\, \mathrm{d}t.
\end{align*}
From the definition of $F$ and Proposition \ref{prop:N_estimates}(i)-(ii), it follows that
\begin{align*}
&   \left| \int _0^T \langle F(\by_n(t))-F(\bar{\by}(t)),\bz(t) \rangle _{Y} \, \mathrm{d}t \right| \\
&\qquad  \le \int_0^T |\langle N(\by_n(t)-\bar{\by}(t),\by_n),\bz(t) \rangle_{Y}|\, \mathrm{d}t + \int_0^T |\langle N(\bar{\by}(t),\by_n(t)-\bar{\by}(t)),\bz(t) \rangle_{V',V}| \,\mathrm{d}t \\
 & \qquad \le\int _0^T \| N(\by_n(t)-\bar{\by}(t),\by_n(t))\|_{Y} \|\, \bz(t)\|_Y\, \mathrm{d}t +
 \int _0^T \| N(\bar{\by}(t),\by_n(t)-\bar{\by}(t))\|_{Y} \|\, \bz(t)\|_Y\, \mathrm{d}t \\
 &\qquad \le M\int_0^T \| \by_n(t)-\bar{\by}(t)\|_V \, \| \by_n(t)\|_{\bbH^{2}(\Omega)} \, \| \bz(t)\|_Y \, \mathrm{d}t + M
 \int_0^T \| \bar{\by}(t)\|_{\bbH^{2}(\Omega)}\, \| \by_n(t)-\bar{\by}(t)\|_V \, \| \bz(t) \| _Y \, \mathrm{d}t \\
 &\qquad \le M\left( \| \by_n\|_{L^2(0,T;\mD(A_\lambda))} \, \| \by_n - \bar{\by} \|_{L^2(0,T;V)} \, \| \bz \| _{L^{\infty}(0,T;Y)} +
 \| \bar{\by}\|_{L^2(0,T;\mD(A_\lambda))} \, \| \by_n - \bar{\by} \|_{L^2(0,T;V)} \, \| \bz \| _{L^\infty(0,T;Y)} \right).
\end{align*}
Since $\mD(A_\lambda)$ is compactly embedded in $V$, by the Aubin-Lions lemma it follows that $\| \by_n -\bar{\by}\|_{L^2(0,T;V)} \underset{n \to \infty}{\longrightarrow} 0$. Passing to the limit in \eqref{eq:lim_state_eq} yields
\begin{align*}
\int_0^T \left\langle \frac{\mathrm{d}}{\mathrm{d}t} \bar{\by}(t),\bz(t) \right\rangle _{Y}\mathrm{d}t= \int_0^T \langle A\bar{\by}  (t) - F(\bar{\by}(t))+B \bar{u}(t),\bz(t) \rangle _{Y}\, \mathrm{d}t.
\end{align*}
Since $H^{1}(0,T;Y)$ is dense in $L^{2}(0,T;Y)$, it holds that $e(\bar{\by},\bar{u})=(0,\by_0)$. From weak lower semi-continuity of norms we finally obtain
\begin{align*}
  J(\bar{\by},\bar{u}) \le \underset{n \to \infty}{\liminf} \ J(\by_n,u_n).
\end{align*}
This shows that $(\bar{\by},\bar{u})$ is optimal.

Note that the bound \eqref{eq:sol_est_NLQ} can be shown to hold for arbitrary optimal solutions since they necessarily have to satisfy \eqref{eq:aux_bound_opt} from which we can argue as above.
\end{proof}

\begin{remark}\label{rem:lack_diff_on_Y}
For the previous proof, Corollary \ref{cor:more_regularity} was essential. It is not available in dimension 3 with $\by_0 \in Y$ and $\by \in W_\infty(V,V')$. For this reason, we cannot expect differentiability of $\mV$ on $Y$.
\end{remark}


\begin{lemma}\label{lem:contractive_pert}
   Let $G \in \mathcal{L}(W_\infty(\mD(A_\lambda),Y),L^2(0,\infty;Y))$ be such that $\| G\| < \frac{1}{M_K},$ where $\|G\|$ denotes the operator norm of $G$. Then, for all $\bff \in L^2(0,\infty;Y)$ and $\by_0\in V$, there exists a unique solution to the following system:
  \begin{align*}
    \dot{\by} = (A-BK)\by(t) + (G\by)(t) + \bff(t), \ \ \by(0)=\by_0.
  \end{align*}
Moreover,
\begin{align*}
  \| \by \| _{W_\infty(\mD(A_\lambda),Y)} \le \frac{M_K}{1-M_K \| G\|} ( \| \bff \| _{L^2(0,\infty;Y)} + \| \by_0 \| _V).
\end{align*}
\end{lemma}
\begin{proof}
The assertion is a variant of \cite[Lemma 2.5]{BreKP18} and follows by the arguments provided in the latter reference.
\end{proof}

\begin{proposition} \label{proposition:optiCondWeak}
There exists $\delta_{2} \in (0,\delta_1]$ such that for all $\by_0\in B_V(\delta_2)$, and for all solutions $(\bar{\by},\bar{u})$ of  \eqref{eq:NLQprob}, there exists a unique costate $\bp \in W_\infty(Y,[\mD(A_\lambda)]')$ satisfying
\begin{align} \label{eq:costate_NLQ}
-\dot{\bp} - A'\bp - P((\bar{\by}\cdot \nabla)\bp-(\nabla \bar{\by})^T \bp)&=\bar{\by}  \quad \text{(in $L^2(0,\infty;[\mD(A_\lambda)]')$)}, \\
\label{eq:control_NLQ}
\alpha \bar{u} + B^* \bp = \ & 0.
\end{align}
Moreover, there exists a constant $M>0$, independent of $(\bar{\by},\bar{u})$, such that
\begin{equation} \label{eq:estim_costate_NLQ}
\| \bp \|_{W_\infty(Y,[\mD(A_\lambda)]')}   \le M\left(\|\bar{\by} \|_{L^2(0,\infty;Y)}+ \alpha \|\bar{u}\|_{L^2(0,\infty;U)}  \right) .
\end{equation}
\end{proposition}

\begin{remark}\label{rem:very_weak_eq}
{\em {Equation \eqref{eq:costate_NLQ} is satisfied in the sense that
\begin{equation}\label{eq:costate_interp}
\langle \bp, \dot{\bz} \rangle_{L^2(0,\infty;Y)} - \langle \bp,A \bz - P((\bar \by\cdot \nabla) \bz +(\bz \cdot \nabla) \bar \by ) \rangle_{L^2(0,\infty;Y)} =  \langle \bar \by, \bz\rangle_{L^2(0,\infty;Y)},
\end{equation}
for all $\bz\in W^0_\infty(\mD(A_\lambda),Y)$, where
\begin{equation}\label{eq:aux_space_comp_supp}
W^0_\infty(\mD(A_\lambda),Y)=\left\{ \bz \in W_\infty(\mD(A_\lambda),Y) \ | \ \bz(0)=0 \right\}.
\end{equation}

}}
\end{remark}

\begin{proof}
  For now, let us assume that $\delta_2=\delta_1$. Then
   problem \eqref{eq:NLQprob} has a solution $(\bar{\by},\bar{u})$ by Lemma \ref{lem:ex_opt_sol}.
We are going to derive optimality conditions by proving that the linearization of $e$ is surjective. Since $F(y)=PN(\by,\by)$ and by Corollary \ref{cor:N_time_estimates}, we conclude that
$N$ and $F$ are Fr\'{e}chet differentiable. Hence, $e$ is Fr\'{e}chet differentiable with
\begin{align*}
 &De(\by,u)\colon W_\infty(\mD(A_\lambda),Y) \times L^2(0,\infty;U) \to L^2(0,\infty;Y)\times V\\
 &De(\by,u)(\bz,v)=(\dot{\bz}-(A\bz-P(N(\by,\bz)+N(\bz,\by))+Bv),\bz(0)).\end{align*}
We are going to show that $De(\bar{\by},\bar{u})$ is surjective for $\delta_{2}$ small enough. For an arbitrary pair $(\mathbf{r},\mathbf{s})\in L^2(0,\infty;Y) \times V$, we consider
\begin{align*}
\dot{\bz} - (A\bz-P(N(\bar{\by},\bz)+N(\bz,\bar{\by}))+Bv) &= \mathbf{r}, \ \
\bz(0)  = \mathbf{s}.
\end{align*}
From Corollary \ref{cor:N_time_estimates} and Lemma \ref{lem:ex_opt_sol}, it follows that
\begin{align*}
 \| P(N(\bar{\by},\bz)+N(\bz,\bar{\by}))\|_{L^2(0,\infty;Y)}
 &\le M \| \bar{\by}\|_{W_\infty(\mD(A_\lambda),Y)} \| \bz\|_{W_\infty(\mD(A_\lambda),Y)} \le M \delta_{2} \| \bz\|_{W_\infty(\mD(A_\lambda),Y)}.
\end{align*}
By possibly reducing $\delta_{2}$, we can w.l.o.g.\@ assume that
the operator $G\in \mathcal{L}(W_\infty(\mD(A_\lambda),Y),L^2(0,\infty;Y))$ defined by
\begin{align*}
 (G\bz)(t) := DF(\by(t))(\bz(t))= P(N(\bar{\by}(t),\bz(t))+N(\bz(t),\bar{\by}(t)))
\end{align*}
satisfies $\|G\| < \frac{1}{M_K}$. Lemma \ref{lem:contractive_pert} implies the existence of a unique solution $\bz \in W_{\infty}(\mD(A_\lambda),Y)$ to  the system
\begin{align*}
 \dot{\bz}- ((A-BK)\bz+P(N(\bar{\by},\bz)+N(\bz,\bar{\by}))) &= \mathbf{r}, \ \
\bz(0)  = \mathbf{s}.
\end{align*}
  The surjectivity of $De(\bar{\by},\bar{u})$ follows by setting  $v=-K\bz\in L^2(0,\infty;U)$. Moreover, we additionally have that
\begin{align}\label{eq:nonlin_aux1}
\| \bz \| _{W_\infty(\mD(A),Y)} \le M \left( \| \mathbf{r} \| _{L^2(0,\infty;Y)}+ \| \mathbf{s} \| _V\right)
\end{align}
where the constant $M$ is independent of $(\mathbf{r},\mathbf{s})$ and $\by_0$.
The surjectivity of $De(\bar{\by},\bar{u})$ implies the existence of a unique pair $(\bp,\mu)\in L^2(0,\infty;Y)\times V'$ such that for all $(\bz,v)\in W_\infty(\mD(A_\lambda),Y)\times L^2(0,\infty;U)$
\begin{equation}\label{eq:nonlin_aux2}
DJ(\bar{\by},\bar{u})(\bz,v) - \langle (\bp,\mu), De(\bar{\by},\bar{u})(\bz,v) \rangle_{L^2(0,\infty;Y)\times V',L^2(0,\infty;Y)\times V}=0.
\end{equation}
In the following, we use \eqref{eq:nonlin_aux2} to derive the costate equation \eqref{eq:costate_NLQ} and relation \eqref{eq:control_NLQ}.
Note that $J$ is differentiable with
\begin{equation}\label{eq:nonlin_aux3}
\begin{aligned}
 DJ(\bar{\by},\bar{u})(\bz,v)  &= \langle \bar{\by},\bz \rangle_{L^2(0,\infty;Y)} + \alpha \langle \bar{u},v \rangle_{L^2(0,\infty;U)}= \Big\langle \begin{pmatrix} \bar{\by}  \\ \alpha \bar{u}  \end{pmatrix},
\begin{pmatrix}  \bz \\ v \end{pmatrix} \Big\rangle_{L^2(0,\infty;Y) \times L^2(0,\infty;U)}.
\end{aligned}
\end{equation}
Moreover, for all $(\bz,v)\in W_\infty(\mD(A_\lambda),Y) \times L^2(0,\infty;U)$
\begin{equation}\label{eq:nonlin_aux4}
\begin{aligned}
  &\langle (\bp,\mu),De(\bar{\by},\bar{u})(\bz,v)) \rangle_{L^2(0,\infty;Y)\times V',L^2(0,\infty;Y)\times V} \\&= \langle \bp,\dot{\bz}\rangle_{L^2(0,\infty;Y)} - \langle \bp,A \bz \rangle _{L^2(0,\infty;Y)} + \langle  \bp, G \bz  \rangle _{L^2(0,\infty;Y)}- \langle \bp,Bv \rangle_{L^2(0,\infty;U)}  + \langle \mu, \bz(0) \rangle_{V',V}.
\end{aligned}
\end{equation}
Taking $\bz=0$ and letting $v$ vary in $L^2(0,\infty;U)$, from \eqref{eq:nonlin_aux2}, \eqref{eq:nonlin_aux3} and \eqref{eq:nonlin_aux4}, we obtain
\begin{align*}
 \alpha \bar{u} + B^* \bp = 0 \text{ in } L^2(0,\infty;U),
\end{align*}
and, hence, relation \eqref{eq:control_NLQ}.

Taking  $v=0$, we conclude that for all $\bz \in W^0_\infty(\mD(A_\lambda),Y)$, we have
\begin{equation}\label{eq:nonlin_aux5}
 \langle \bp, \dot{\bz} \rangle_{L^2(0,\infty;Y)} = \langle \bp,A \bz -G \bz\rangle_{L^2(0,\infty;Y)} + \langle \bar \by, \bz\rangle_{L^2(0,\infty;Y)} .
\end{equation}
With Proposition \ref{prop:N_estimates}(i) and (ii) we obtain for all $z\in W_\infty^0(\mD(A_\lambda),Y)$
\begin{equation*}
\begin{aligned}
&\langle \bp,\dot{\bz} \rangle_{L^2(0,\infty;Y)}\le | \langle \bp,A\bz\rangle _{L^2(0,\infty;Y)} | +| \langle \bp,G\bz \rangle_{L^2(0,\infty;Y)} | + | \langle \bar{\by},\bz \rangle_{L^2(0,\infty;Y)} | \\
&\qquad \le M \left( \| \bp \|_{L^2(0,\infty;Y)}  +\| \bar{\by} \|_{L^2(0,\infty;Y)} \right) \| \bz \|_{L^2(0,\infty;\bbH^2(\Omega))} + |\langle \bp, P(N(\bar \by,\bz) + N(\bz, \bar \by))\rangle_{L^2(0,\infty;Y)}| \\
&\qquad \le M \left( \| \bp \|_{L^2(0,\infty;Y)}  +\| \bar{\by} \|_{L^2(0,\infty;Y)} \right) \| \bz \|_{L^2(0,\infty;\bbH^2(\Omega))}+\int_0^\infty |\langle \bp(t), N(\bar \by(t),\bz(t)) + N(\bz(t), \bar \by(t))\rangle_Y| \, \mathrm{d}t \\
&\qquad \le M \left( \| \bp \|_{L^2(0,\infty;Y)}  +\| \bar{\by} \|_{L^2(0,\infty;Y)} \right) \| \bz \|_{L^2(0,\infty;\bbH^2(\Omega))} +  M \int_0^\infty \|\bp(t)\|_Y \, \|\bar \by(t)\|_{V} \, \|\bz(t)\|_{\bbH^2(\Omega)}\, \mathrm{d}t\\
&\qquad \le M \left( \| \bp \|_{L^2(0,\infty;Y)}  +\| \bar{\by} \|_{L^2(0,\infty;Y)} + \| \bp \|_{L^2(0,\infty;Y)} \| \bar{\by} \|_{L^\infty(0,\infty;V)} \right)  \| \bz \|_{L^2(0,\infty;\bbH^2(\Omega))} \\
&\qquad \le M \left( \| \bp \|_{L^2(0,\infty;Y)}  +\| \by_0 \|_{V}   \right)  \| \bz \|_{L^2(0,\infty;\bbH^2(\Omega))}.
\end{aligned}
\end{equation*}
Since $W^0_\infty(\mD(A_\lambda),Y)$ is dense in $L^2(0,\infty;\bbH^2(\Omega))$ for the $L^2(0,\infty;\bbH^2(\Omega))$-norm, we can extend $\dot{\bp}$ to a bounded linear form on $L^2(0,\infty;\bbH^2(\Omega))$, i.e., $\dot{\bp}$ can be extended to an element of $L^2(0,\infty;[\mD(A_\lambda)]')$, moreover the following bound holds true:
\begin{align}\label{eq:nonlin_aux6}
\| \dot{\bp} \|_{L^2(0,\infty;[\mD(A_\lambda)]')} \le M \left( \| \bp \| _{L^2(0,\infty;\bbH^2(\Omega))} + \| \by _0 \| _V \right).
\end{align}
It follows that $\bp \in W_\infty(Y,[\mD(A_\lambda)]')$ and that the costate equation \eqref{eq:costate_NLQ} is satisfied.

Let us bound $\bp \in L^2(0,\infty;Y)$. For this purpose, consider $\mathbf{r} \in L^2(0,\infty;Y)$ and assume that $(\bz,v)$ satisfy  $De(\bar{\by},\bar{u})(\bz,v)=(\mathbf{r},0)$ and the bound \eqref{eq:nonlin_aux1} (with $\mathbf{s}=0$). From \eqref{eq:nonlin_aux2}, the expression \eqref{eq:nonlin_aux3} of $DJ(\bar{\by},\bar{u})$, estimate \eqref{eq:nonlin_aux1}, and estimate \eqref{eq:sol_est_NLQ} on $(\bar{\by},\bar{u})$, we derive:
\begin{align*}
 \langle \bp, \mathbf{r} \rangle _{L^2(0,\infty;Y)}& = \langle (\bp,\mu),(\mathbf{r},0) \rangle_ {L^2(0,\infty;Y) \times V', L^2(0,\infty;Y) \times V}  \\
& =\langle De(\bar{\by},\bar{u})'(\bp,\mu),(\bz,v) \rangle_{W_\infty(\mD(A_\lambda),Y)' \times L^2(0,\infty;U), W_\infty(\mD(A_\lambda),Y) \times L^2(0,\infty;U) } \\
& = DJ(\bar{\by},\bar{u})(\bz,v) \\
& \le M\left( \| \bar{\by} \|_{L^2(0,\infty;Y)} + \| \bar{u} \|_{L^2(0,\infty;U)} \right) \left( \| \bz \|_{L^2(0,\infty;Y)} + \| v \|_{L^2(0,\infty;U)} \right) \\
&\le M \| \by _0 \| _V \| \| \mathbf{r} \| _{L^2(0,\infty;Y)}.
\end{align*}
Since $\mathbf{r}$ was arbitrary and $M$ is independent of $\mathbf{r}$, we obtain that $\| \bp \| _{L^2(0,\infty;Y)} \le M \| \by _0 \| _V$. Combining this estimate with \eqref{eq:nonlin_aux6}, we finally obtain \eqref{eq:estim_costate_NLQ}.
\end{proof}

\subsection{Sensitivity analysis with respect to $\by_0$}

We define the space
\begin{align}\label{eq:X_prod_space}
	X:= V\times L^2(0,\infty;Y) \times L^2(0,\infty;[\mD(A_\lambda)]') \times L^2(0,\infty;U),
\end{align}
endowed with the $l_\infty$ product norm. Consider now the mapping $\Phi$, defined from $W_\infty(\mD(A_\lambda),Y) \times L^2(0,\infty;U) \times W_\infty(Y,[\mD(A_\lambda)]')$ to $X$ by
\begin{equation} \label{eq:defPhi}
\Phi(\by,u,\bp) =
\begin{pmatrix}
\by(0) \\ \dot{\by}- A\by + F(\by)-Bu \\ - \dot{\bp} - A'\bp - P((\by\cdot \nabla)\bp-(\nabla \by)^T \bp) - \by \\
\alpha u + B^* \bp
\end{pmatrix}.
\end{equation}

The well-posedness of $\Phi$ follows from the considerations on $e(\by,u)$ and the costate equation \eqref{eq:costate_NLQ} that have been given in the proof of Proposition \ref{proposition:optiCondWeak}.

\begin{lemma} \label{lemma:inverseMappingWeak}
There exist $\delta_3 > 0$, $\delta_3'>0$, and three $C^\infty$-mappings
\begin{equation*}
\by_0 \in B_V(\delta_3) \mapsto \big( \mathcal{Y}(\by_0),\mathcal{U}(\by_0),\mathcal{P}(\by_0) \big)
\in W_\infty(\mD(A_\lambda),Y) \times L^2(0,\infty;U) \times W_\infty(Y,[\mD(A_\lambda)]')
\end{equation*}
such that for all $\by_0 \in B_V(\delta_3)$, the triplet $\big( \mathcal{Y}(\by_0),\mathcal{U}(\by_0),\mathcal{P}(\by_0) \big)$ is the unique solution to
\begin{equation}\label{eq:inverse_aux}
\Phi(\by,u,\bp)= (\by_0,0,0,0), \quad
\max \big( \| \by \|_{W_\infty(\mD(A_\lambda),Y)}, \| u \|_{L^2(0,\infty;U)}, \| \bp \|_{W_\infty(Y,[\mD(A_\lambda)]')} \big) \leq \delta_3'
\end{equation}
in $W_\infty(\mD(A_\lambda),Y) \times L^2(0,\infty;U) \times W_\infty(Y,[\mD(A_\lambda)]')$.
Moreover, there exists a constant $M>0$ such that for all $\by_0 \in B_V(\delta_3)$,
\begin{equation} \label{eq:lipschitzReg}
\max \big( \| \mathcal{Y}(\by_0) \|_{W_\infty(\mD(A_\lambda),Y)}, \| \mathcal{U}(\by_0) \|_{L^2(0,\infty;U)}, \| \mathcal{P}(\by_0) \|_{W_\infty(Y,[\mD(A_\lambda)]')} \big) \leq M \| \by_0 \|_V.
\end{equation}
\end{lemma}

\begin{proof}
Let us show the statement by means of the
inverse function theorem. For this purpose, note that $\Phi$ only contains polynomial terms and thus is infinitely differentiable. It further holds that $\Phi(0,0,0)= (0,0,0,0)$. Let us  prove that $D\Phi(0,0,0)$ is an isomorphism and take  $(\bw_1,\bw_2,\bw_3,w_4) \in X$ and $(\by,u,\bp)\in W_\infty(\mD(A_\lambda),Y)\times L^2(0,\infty;U) \times W_\infty(Y,[\mD(A_\lambda)]')$. We have
\begin{equation} \label{eq:equivalenceWeak}
D\Phi(0,0,0) (\by,u,\bp)= (\bw_1,\bw_{2},\bw_{3},w_4) \Longleftrightarrow
\begin{cases}
\begin{array}{rcl}
\by(0) & = & \bw_1 \\
\dot{\by} - A \by - Bu & = & \bw_2 \\
-\dot{\bp} - A' \bp - \by & = & \bw_3 \\
\alpha u + B^* \bp & = & w_4.
\end{array}
\end{cases}
\end{equation}
By Proposition \ref{prop:linOS}, the linear system on the left-hand has a unique solution $(\by,u,\bp)$, moreover
\begin{align*}
\| (\by,u,\bp) \|_{W_\infty(\mD(A_\lambda),Y) \times L^2(0,\infty;U) \times W_\infty(Y,[\mD(A_\lambda)]')}
\leq \ & M \| (\bw_1,\bw_2,\bw_3,w_4) \|_X .
\end{align*}
Hence, $D\Phi(0,0,0)$ is an isomorphism and, with the inverse function theorem, we have  $\delta_3>0$, $\delta_3'>0$, and $C^\infty$-mappings $\mathcal{Y}$, $\mathcal{U}$, and $\mathcal{P}$ satisfying \eqref{eq:inverse_aux}.

With regard to estimate \eqref{eq:lipschitzReg}, let us (possibly) reduce $\delta_3$ such that the norms of the derivatives of the three mappings are bounded on $B_V(\delta_3)$ by some constant $M>0$. As a consequence, the three mappings are in particular Lipschitz continuous with modulus $M$ and estimate \eqref{eq:lipschitzReg} follows from the fact that $\big( \mathcal{Y}(0),\mathcal{U}(0),(\mathcal{P}(0) \big)= (0,0,0)$.
\end{proof}


\begin{proposition} \label{proposition:UisOptimal}
There exists $\delta_4\in (0,\min(\delta_2,\delta_3)]$ such that for all $\by_0 \in B_V(\delta_4)$, the pair $(\mathcal{Y}(\by_0),\mathcal{U}(\by_0))$ is the unique solution to \eqref{eq:NLQprob} with initial condition $\by_0$. Moreover,  $\mathcal{P}(\by_0)$ is the unique associated costate.
\end{proposition}
\begin{proof}
 For now, assume that $\delta_4= \min(\delta_2,\delta_3)$  and consider  $\by_0 \in B_V(\delta_4)$. Then, there exists a solution $(\bar{\by},\bar{u})$ to \eqref{eq:NLQprob} with associated costate $\bar{\bp}$ due to Lemma \ref{lem:ex_opt_sol} and Proposition \ref{proposition:optiCondWeak}. This solution has to satisfy
 \begin{align*}
 \max ( \| \bar{\by} \|_{W_\infty(\mD(A_\lambda),Y)}, \| \bar{u} \|_{L^2(0,\infty;U)}, \| \bar{\bp} \| _{W_\infty(Y,[\mD(A_\lambda)]')} ) \le M \| \by _0 \|_V.
\end{align*}
A (possible) reduction of $\delta_4$ ensures that
\begin{align*}
 \max ( \| \bar{\by} \|_{W_\infty(\mD(A_\lambda),Y)}, \| \bar{u} \|_{L^2(0,\infty;U)}, \| \bar{\bp} \| _{W_\infty(Y,[\mD(A_\lambda)]')} ) \le \delta_3'.
\end{align*}
Since $\Phi(\bar{\by},\bar{u},\bar{\bp} ) = (\by_0,0,0,0)$, Lemma \ref{lemma:inverseMappingWeak} implies that $(\bar{\by},\bar{u},\bar{\bp})=(\mathcal{Y}(\by_0),\mathcal{U}(\by_0),\mathcal{P}( \by_0) )$ which proves the assertion.
\end{proof}

\begin{corollary} \label{coro:diff}
The value function $\mathcal{V}$ is infinitely differentiable on $B_V(\delta_4)$.
\end{corollary}
\begin{proof}
Note that by definition we have $\mathcal{V}(\by_0) = J(\mathcal{Y}(\by_0),\mathcal{U}(\by_0))$. Since $J$ is infinitely differentiable, differentiability of $\mV$ is a consequence of the composition of infinitely differentiable mappings.
\end{proof}


\begin{lemma} \label{lemma:costateEqualDerivative}
There exists $\delta_5 \in (0,\delta_4]$ such that for all $\by_0 \in B_V(\delta_5)$,
$\| \mathcal{Y}(\by_0) \|_{L^\infty(0,\infty;V)} \leq  \delta_4$
and
\begin{equation}\label{eq:costate_derValueFunc}
\bp(t)= D\mathcal{V}(\by(t)), \quad \forall t \geq 0 \ \  (\text{in } V')
\end{equation}
where $\by= \mathcal{Y}(\by_0)$ and $\bp= \mathcal{P}(\by_0)$. Additionally, it holds that $D\mV(\by(\cdot))\in L^2(0,\infty;Y)$.
\end{lemma}
The proof is given in Appendix \ref{sec:appendix_proofs}.



Using the optimality condition \eqref{eq:control_NLQ}, we obtain the optimal control in feedback form.

\begin{corollary} \label{coro:optimFeedback}
For all $\by_0 \in B_V(\delta_5)$,
\begin{equation*}
u(t)= -\frac{1}{\alpha} B^{*}D\mathcal{V}(\by(t)), \quad \text{for a.e.\@ $t > 0$},
\end{equation*}
where $\by= \mathcal{Y}(\by_0)$ and $u= \mathcal{U}(\by_0)$.
\end{corollary}

In the two-dimensional case, see \cite[Proposition 16]{BreKP19}, an approximation of the value function was
obtained by investigating the equations which can be derived  by
 successive differentiation   of  the Hamilton-Jacobi-Bellman equation for the optimal solution with respect to $\mathcal{V}$. The HJB equation associated to \eqref{eq:NLQprob} is given by
\begin{align}\label{eq:HJB}
 D\mV(\by)(A\by-F(\by))+\frac{1}{2}\| \by \|_Y^2 - \frac{1}{2\alpha} \| B^* D\mV(\by)\|_U^2 = 0 \text{ for } \by\in \mD(A).
\end{align}
In the 2-D case this equation is rigorous, in the 3-D case, however, it is only formal. In fact,
the term $B^* D\mV(\by)$ is not well-defined for $B \in \mathcal{L}(U,Y)$ since $\mV$ is not differentiable on $Y$.




\section{Multilinear Lyapunov operator equations}

The purpose of this section is to analyze a sequence of certain multilinear operator equations. In the two-dimensional case, in \cite{BreKP19b} it is shown that these equations can be derived by successive differentiations of the HJB equation. Moreover, their solutions are multilinear forms that represent the derivatives of the value function in zero. In \cite{BreKP19b}, this enabled us to derive a polynomial feedback law via a Taylor series expansion of the value function. In contrast to the 2-D case, at this point we cannot follow the arguments provided in \cite{BreKP19b}. In particular, so far we only know that $\mV$ is differentiable on $B_{V}(\delta_{4})$ but not necessarily on $B_{Y}(\delta_{4})$. As will be shown below, it is nevertheless possible to derive a unique sequence of multilinear forms that result in a polynomial feedback law which locally approximates the optimal control.

Let us begin with the algebraic operator Riccati equation
\begin{align}\label{eq:are}
 \langle \bz_{2},A^{*} \Pi \bz_{1 } \rangle_{Y} + \langle \Pi A\bz_{1},\bz_{2}\rangle _{Y}+\langle \bz_{2},\bz_{1} \rangle_{Y} - \frac{1}{\alpha} \langle B^{*} \Pi \bz_{1},B^{*}\Pi \bz_{2} \rangle_{U} = 0, \ \ \forall \bz_1,\bz_2 \in \mD(A).
\end{align}
A general treatment of \eqref{eq:are} for abstract linear control problems has been given in, e.g.,  \cite{CurZ95,LasT00}. We emphasize that the stabilizability assumption \ref{ass:A2} and \eqref{eq:VY-coerc}, which implies exponential detectability of $(A,\mathrm{id})$, ensure the existence of a unique nonnegative stabilizing solution $\Pi \in \mathcal{L}(Y)$ to \eqref{eq:are}.
 In the following, we use the notation
$$A_{\pi} := A-\frac{1}{\alpha}BB^{*}\Pi$$
for the closed-loop operator associated with the linearized stabilization problem. Since $A_{\pi}$ generates an analytic exponentially stable semigroup $e^{A_{\pi}t}$ on $Y$, for trajectories of the form $\tilde{\by}(\cdot)=e^{A_\pi\cdot}\by, \by\in V$ it follows that $\tilde{\by}\in W_{\infty}(\mD(A_\pi),Y)$. Similarly, for $\by \in V'$,  \cite[Corollary II.3.2.1]{Benetal07}, we have that $\tilde{\by}(\cdot)=e^{A\cdot}\by \in W_{\infty}(Y,[\mD(A_\lambda^*)]')$. Note, that \cite[Corollary II.3.2.1]{Benetal07} is stated for $T<\infty$. But we can follow its proof and apply \cite[Theorem II.1.3.1(i)]{Benetal07} instead of \cite[Theorem II.1.3.1(ii)]{Benetal07} to arrive at the conclusion for $T=\infty$.


From Corollary \ref{coro:diff} we already know that $\mathcal{V}$ is infinitely differentiable on $B_V(\delta_4)$. We shall relate its $k$-th derivative $D^k\mV(0)$ in zero to a multilinear form $\mathcal{T}_{k}\in \mathcal{M}(V^k,\mathbb R)$. Below, we study such multilinear forms and show that they additionally satisfy
\begin{align}\label{eq:more_reg_ex}
\mathcal{T}_{k} \in  \mathcal{S}_k(V,V'):=\bigcap\limits_{\ell=1}^k \mM(V^{\ell-1}\times V'\times V^{k-\ell},\mathbb R),
\end{align}
i.e., the form $\mT_k$ can be extended  from $V$ to $V'$ as  bounded multilinear form separately in each coordinate. For $\mT_k\in S_k(V,V')$, we introduce
\begin{align*}
 |\mT_k| = \sum_{\ell=1}^k \sup\limits_{\substack{\| \bv_{\ell} \|_{V'}\le 1\\ \| \bv_i \|_V \le 1, i\neq \ell }} | \mT_k(\bv_1,\dots,\bv_{\ell-1},\bv_{\ell},\bv_{\ell+1},\dots,\bv_k ) |.
\end{align*}
Let us illustrate the definition of $S_k(V,V')$ in the case that $k=2$. For $\mathcal{T}_{2}\in \mM(V\times V,\mathbb R)$, we can define an associated operator $\widetilde{\Pi}$ via
\begin{align*}
\widetilde{\Pi}&\colon V \to V', \quad
\widetilde{\Pi}\colon \bv \mapsto \widetilde{\Pi} \bv := \mathcal{T}_{2}(\bv,\cdot) \ \ \text{for all } \bv \in V.
\end{align*}
 Indeed, we have $\widetilde{\Pi} \in \mL(V,V')$ since
\begin{align*}
 \| \widetilde{\Pi} \bv \| _{V'} = \sup\limits_{\| \tilde{\bv} \|_V \le 1} | \langle \tilde{\bv}, \widetilde{\Pi}\bv\rangle_{V,V'}|= \sup\limits_{\| \tilde{\bv} \|_V \le 1}|\mathcal{T}_{2}(\bv,\tilde{\bv})| \le M   \| \bv \| _V.
\end{align*}
  The following considerations further clarify the additional regularity which can be gained if  $\mathcal{T}_{2} \in S_2(V,V')$, rather than $\mathcal{M}(V\times V,\R)$ only.  For  $\mathcal{T}_{2} \in S_2(V,V')$ and hence $\mathcal{T}_{2} \in \mathcal{M}(V\times V',\R)$ we have that $\bv_1\to\mathcal{T}_{2}(\bv_1,\cdot) $ can be represented by an operator $\Pi_1 \in \mathcal{L}(V)$ such that
   \begin{equation}\label{eq:aux11}
  \mathcal{T}_{2}(\bv_1,\bv_2)=\langle \bv_2,\Pi_1\bv_1\rangle_{V',V} \text{ for all } \bv_1 \in V, \bv_2\in V'.
  \end{equation}
  Similarly, since $\mathcal{T}_{2} \in \mathcal{M}(V'\times V,\R)$ we have that  $\bv_1\to\mathcal{T}_{2}(\bv_1,\cdot) $ can be represented by an operator $\Pi_2 \in \mathcal{L}(V')$ such that
  $$
  \mathcal{T}_{2}(\bv_1,\bv_2)=\langle \bv_2,\Pi_2\bv_1\rangle_{V,V'} \text{ for all } \bv_1 \in V', \bv_2\in V.
  $$
  For $(\bv_1,\bv_2)\in V\times V$ we have from \eqref{eq:aux11}
 that  $\mathcal{T}_{2}(\bv_1,\bv_2) =\langle \bv_2,\Pi_1\bv_1\rangle_{V',V}=\langle \bv_2,\Pi_2\bv_1\rangle_{V,V'}$, and thus for each $\bv_1\in V$
 $$
 \langle \bv_2,(\Pi_1- \Pi_2)\bv_1\rangle_{V,V'}=0 \text{ for all } \bv_2 \in V.
 $$
 Thus $(\Pi_1-\Pi_2)|_V =0 $, and since $V$ is dense in $V'$, and $\Pi_2\in \mathcal{L}(V')$, the operator $\Pi_2$ is the unique continuous extension of $\Pi_1$ from $V$ to $V'$. We  denote both $\Pi_1$ and $\Pi_2$ by $\Pi$. Since $\Pi \in \mathcal{L}(V)\cap\mathcal{L}(V')$ interpolation implies that $\Pi \in \mathcal{L}(Y)$. For $(\bv_1,\bv_2)\in Y\times V$ we have $\mathcal{T}_2(\bv_1,\bv_2)=\langle \bv_2,\Pi \bv_1\rangle_{V,V'} = \langle \bv_2,\Pi \bv_1\rangle_{Y}$. By continuity in the first coordinate of the inner product in $Y$ therefore $\mathcal{T}_2(\bv_1,\bv_2)= \langle \bv_2,\Pi \bv_1\rangle_{Y}$ for all $(\bv_1,\bv_2)\in Y\times Y$. As a consequence, if  $\mathcal{T}_{2} \in S_2(V,V')$,  then $\mathcal{T}_{2} \in \mathcal{M}(Y^2,\R).$

We turn to our first results on the existence of solutions of generalized Lyapunov equations.

\begin{theorem} \label{thm:existenceUniquenessLyapunov_1}
Let $k \geq 3$ and $\mathcal{F} \in \mathcal{S}_{k-1}(V,V').$
Then, there exists a unique $\mT \in \mathcal{S}_k(V,V')$ such that for all
$(\bz_1,\dots,\bz_k) \in \mD(A)^k$,
\begin{equation} \label{eq:Lyapunov2}
\sum_{i=1}^k \mT(\bz_1,\dots,\bz_{i-1},A_{\pi} \bz_i,\bz_{i+1},\dots,\bz_k) = \mathcal{G}(\bz_1,\dots,\bz_k),
\end{equation}
where:
\begin{equation*}
\mathcal{G}(\bz_1,\dots,\bz_k)
=  \sum_{j=1}^{k-1} \sum_{i=1}^{k-j} \mathcal{F} (A_0(\bz_j,\bz_{j+i}),\bz_1,\dots,\bz_{j-1},\bz_{j+1},\dots,\bz_{j+i-1},\bz_{j+i+1},\dots,\bz_k).
\end{equation*}
Moreover, if $\mathcal{G}$ is symmetric, then $\mT,$ considered as an element of $\mM(V^k,\mathbb R)$, is also symmetric.
\end{theorem}

\begin{proof}
The main idea of the proof relies on an explicit integral representation of the solution to multilinear operator equations. For the Lyapunov operator equation, this can be found in, e.g., \cite[Chapter 4]{CurZ95}, for multilinear equations arising in bilinear control problems, we refer to \cite{BreKP19}.

   For arbitrary $\ell \in \{1,\dots,k\}$ assume that $\bz_\ell \in V'$ and $\bz_j\in V$ for $j\neq \ell$. Let us define
   $$(\tilde{\bz}_1,\dots,\tilde{\bz}_k):=(e^{A_{\pi}t }\bz_1,\dots,e^{A_{\pi} t}\bz_k)$$
   as well as
\begin{equation}\label{eq:int_repr_lyap_eq}
\mT(\bz_1,\dots,\bz_k)
= - \int_0^\infty \mathcal{G}(\tilde{\bz}_1,\dots,\tilde{\bz}_k) \dd t.
\end{equation}
As mentioned before, since $\tilde{\bz}_j$ is the unique solution of $\dot{\by} = A_{\pi} \by$ with $\by(0)=\bz_j$ we have $\tilde{\bz}_j\in W_\infty(\mD(A_\lambda),Y)$ if $j\neq \ell$ and $\tilde{\bz}_{\ell} \in W_\infty(Y,[\mD(A_\lambda^{*})]')$. In particular, it holds that
\begin{align*}
\| \tilde{\bz}_j \|_{W_\infty(\mD(A_\lambda),Y)} \le M \| \bz_j \|_V, \quad \|  \tilde{\bz}_{\ell} \| _{W_\infty(Y,[\mD(A_\lambda^{*})]')} \le M \| \bz_{\ell} \|_{V'}.
\end{align*}
\emph{Well-posedness of $\mT$.} Let us show that $\mT$ as defined in \eqref{eq:int_repr_lyap_eq} satisfies $\mT \in S_k(V,V')$.
%
For this, we first consider terms of the form
\begin{align*}
\int_0^\infty |\mathcal{F} (A_0(\tilde{\bz}_j,\tilde{\bz}_{j+i}),\tilde{\bz}_1,\dots,\tilde{\bz}_{j-1},\tilde{\bz}_{j+1},\dots,\tilde{\bz}_{j+i-1},\tilde{\bz}_{j+i+1},\dots,\tilde{\bz}_k)| \mathrm{d}t
\end{align*}
 for which $j\neq \ell$ and $j+i\neq \ell$. In this case, we obtain
\begin{align*}
&\int_0^\infty |\mathcal{F} (A_0(\tilde{\bz}_j,\tilde{\bz}_{j+i}),\tilde{\bz}_1,\dots,\tilde{\bz}_{j-1},\tilde{\bz}_{j+1},\dots,\tilde{\bz}_{j+i-1},\tilde{\bz}_{j+i+1},\dots,\tilde{\bz}_k)|\, \mathrm{d}t \\
& \qquad \le M \int_0^\infty \| A_0(\tilde{\bz}_j,\tilde{\bz}_{j+i}) \| _V \| \tilde{\bz}_{\ell} \|_{V'} \left( \prod\limits_{m \notin \{j,j+i,\ell\}} \| \tilde{\bz}_m \| _V \right) \, \mathrm{d}t \\
& \qquad \le M \| \tilde{\bz}_{\ell} \|_{L^\infty(0,\infty;V')} \left( \prod\limits_{m\notin \{j,j+i,\ell\}} \| \tilde{\bz}_m \|_{L^\infty(0,\infty;V)} \right) \int_0^\infty \| A_0(\tilde{\bz}_j,\tilde{\bz}_{j+i}) \| _V \, \mathrm{d}t  \\
& \qquad \le M \|\bz_{\ell} \| _{V'} \left( \prod\limits_{m\notin \{j,j+i,\ell\}} \| \bz_m \|_V \right)  \int_0^\infty \| A_0(\tilde{\bz}_j,\tilde{\bz}_{j+i}) \| _V \, \mathrm{d}t.
\end{align*}
 Utilizing \eqref{eq:N_L1}  yields
\begin{align*}
\int_0^\infty \| A_0(\tilde{\bz}_j,\tilde{\bz}_{j+i})\|_V \, \mathrm{d}t &\le M   \| \tilde{\bz}_j\|_{L^2(0,\infty;\bbH^2(\Omega))} \| \tilde{\bz}_{j+i} \|_{L^2(0,\infty;\bbH^2(\Omega))} \le M \| \bz_j\| _{V} \| \bz_{j+i}\|_V,
\end{align*}
which shows
\begin{align}\label{eq:lyap_aux1}
	&\int_0^\infty |\mathcal{F} (A_0(\tilde{\bz}_j,\tilde{\bz}_{j+i}),\tilde{\bz}_1,\dots,\tilde{\bz}_{j-1},\tilde{\bz}_{j+1},\dots,\tilde{\bz}_{j+i-1},\tilde{\bz}_{j+i+1},\dots,\tilde{\bz}_k)|\, \mathrm{d}t   \le M\| \bz_{\ell} \| _{V'} \prod \limits_{m \neq\ell } \| \bz_m\|_V  .
\end{align}
For the terms with $j+i=\ell$, we utilize \eqref{eq:N_L1b} to obtain
\begin{equation}\label{eq:lyap_aux2}
\begin{aligned}
&\int_0^\infty |\mathcal{F} (A_0(\tilde{\bz}_j,\tilde{\bz}_{\ell}),\tilde{\bz}_1,\dots,\tilde{\bz}_{j-1},\tilde{\bz}_{j+1},\dots,\tilde{\bz}_{\ell-1},\tilde{\bz}_{\ell+1},\dots,\tilde{\bz}_k)|\, \mathrm{d}t \\
& \quad \le  M  \left( \prod\limits_{m\notin \{j,\ell\}} \| \tilde{\bz}_m \|_{L^\infty(0,\infty;V)} \right) \int_0^\infty \| A_0(\tilde{\bz}_j,\tilde{\bz}_{\ell}) \| _{V'} \, \mathrm{d}t  \\
& \quad \le M \left( \prod\limits_{m\notin \{j,\ell\}} \| \bz_m \|_V \right) \| \tilde{\bz}_j\| _{L^2(0,\infty;\bbH^2(\Omega))} \| \tilde{\bz}_{\ell}\|_{L^2(0,\infty;Y)}  \le M \| \bz_{\ell} \| _{V'} \prod \limits_{m \neq\ell } \| \bz_m\|_V  .
\end{aligned}
\end{equation}
Analogue estimates based on \eqref{eq:N_L1c} yield
\begin{align}\label{eq:lyap_aux3}
\int_0^\infty |\mathcal{F} (A_0(\tilde{\bz}_{\ell},\tilde{\bz}_{\ell+i}),\tilde{\bz}_1,\dots,\tilde{\bz}_{\ell-1},\tilde{\bz}_{\ell+1},\dots,\tilde{\bz}_{\ell+i-1},\tilde{\bz}_{\ell+i+1},\dots,\tilde{\bz}_k)|\, \mathrm{d}t  \le M \| \bz_{\ell} \| _{V'} \prod \limits_{m \neq\ell } \| \bz_m\|_V  .
\end{align}
From \eqref{eq:lyap_aux1}, \eqref{eq:lyap_aux2}, \eqref{eq:lyap_aux3} together with \eqref{eq:int_repr_lyap_eq} we obtain
\begin{align*}
 | \mT (\bz_1,\dots,\bz_{\ell},\dots,\bz_k) | \le M \frac{k(k-1)}{2} \| \bz_{\ell} \| _{V'} \prod \limits_{m \neq\ell } \| \bz_m\|_V  .
\end{align*}
Since $\ell$ was arbitrary this shows that $\mT \in \mathcal{S}_k(V,V')$.

\emph{$\mT$ is a solution}. Let us now show that $\mT$ solves \eqref{eq:Lyapunov2}. For $(\bz_1,\dots,\bz_k)\in \mD(A^2)^k$ we define
\begin{align*}
  f_{ij}\colon t\in [0,\infty) \mapsto \mathcal{F}(A_0(\tilde{\bz}_j,\tilde{\bz}_{j+i}),\tilde{\bz}_1,\dots,\tilde{\bz}_{j-1},\tilde{\bz}_{j+1},\dots,\tilde{\bz}_{j+i-1},\tilde{\bz}_{j+i+1},\dots,\tilde{\bz}_k)
\end{align*}
 where $j\in \{1,\dots,k-1\}$ and $i\in \{1,\dots,k-j\}$ are arbitrary and $\tilde{\bz}_m(t):=e^{A_\pi t} \bz_m$. By the previous considerations, we know that $f_{ij}\in L^1(0,\infty)$.
Note that we have the  decomposition $f_{ij}(t) = \widehat{\mathcal{F}} \circ \widehat{f}_{ij}(t)$,
where
\begin{align*} \widehat{\mathcal{F}}(\bv_1,\dots,\bv_k)&:=\mathcal{F}(A_0(A_\pi^{-1}\bv_{j},A_\pi^{-1}\bv_{j+i}),\bv_1,\dots,\bv_{j-1},\bv_{j+1},\dots,\bv_{j+i-1},\bv_{j+i+1},\dots,\bv_k),\\
 \widehat{f}_{ij}(t) &:= e^{A_\pi t} (\bz_1,\dots,\bz_{j-1}, A_{\pi} \bz_j,\bz_{j+1},\dots,\bz_{j+i-1}, A_{\pi}\bz_{j+i},\bz_{j+i+1},\dots,\bz_k).
\end{align*}
Observe that since $\mathcal{F}\in S_{k-1}(V,V')$, we have for $(\bv_1,\dots, \bv_k) \in V^k$
\begin{align*}
&| \widehat{\mathcal{F}}(\bv_1,\dots,\bv_k)| \le M \| A_0(A_\pi^{-1} \bv_j,A_\pi^{-1}\bv_{j+i})\|_{V'}  \prod\limits_{m\notin \{j,i+j\}} \| \bv_m\|_V \\
&\quad\le M \| A_\pi^{-1}\bv_j\| _Y \| A_\pi^{-1}\bv_{j+i}\|_{\mathbb H^2(\Omega)}  \prod\limits_{m\notin \{j,j+i\}} \| \bv_m\|_V \le M \| \bv_j\| _{Y} \| \bv_{j+i}\|_{Y}  \prod\limits_{m\notin \{j,i+j\}} \| \bv_m\|_V.
\end{align*}
Hence, $\widehat{\mathcal{F}}\in \mathcal{M}(V^k,\mathbb R)$. Moreover, since $(\bz_,\dots,\bz_k) \in \mD(A^2)^k$ we have $A_\pi e^{A_\pi \cdot} A_\pi \bz_j \in L^1(0,\infty;V)$, see \cite[Theorem II 6.13(c)]{Paz83}, and thus $\widehat{f}_{ij} \in W^{1,1}(0,\infty;V^k)$. From \cite[Lemma 9]{BreKP19}, it follows  that $f_{ij}\in W^{1,1}(0,\infty)$ and
$  f_{ij}'(t) = D\widehat{\mathcal{F}}(\widehat{f}_{ij}(t))\widehat{f}_{ij}'(t).$

Let us now define the function
\begin{align*}
g(t):=\sum\limits_{j=1}^{k-1}\sum\limits_{i=1}^{k-j} f_{ij}(t)= \mathcal{G}(\tilde{\bz}_1,\dots,\tilde{\bz}_k).
\end{align*}
We obviously have $g\in W^{1,1}(0,\infty)$ and
$g(T)-g(0) = \int_0^T g'(t) \mathrm{d}t.$
Moreover, there exists a sequence $\{t_{k}\}$ with $t_{k}\to \infty$ and $|g(t_{k})|\to 0$ for $k\to \infty$. Thus
\begin{align*}
	g(t)=g(t_{k})+ \int_{t_k}^{t}g'(s) \, \mathrm{d}s \Rightarrow |g(t)| \le | g(t_{k})| + \int_{t_{k}}^{\infty} | g'(s) | \, \mathrm{d}s.
\end{align*}
Since $g'\in L^{1}(0,\infty)$, it follows that $\int_{t_{k}}^{\infty} |g'(s)| \, \mathrm{d}s \to 0$ for $t_{k}\to \infty$. This implies $g(t)\to 0$ for $t\to \infty.$  It can be verified that
\begin{align*}
 g'(t) = \sum_{i=1}^k \mathcal{G}(\tilde{ \bz}_1,\dots, \tilde{\bz}_{i-1},A_\pi \tilde{\bz}_i,\tilde{\bz}_{i+1},\dots, \tilde{\bz}_k).
\end{align*}
Asymptotically, this leads to
\begin{align*}
\mathcal{G}(\bz_1,\dots,\bz_k)&=g(0) = - \int_0^\infty g'(t) \mathrm{d}t = -\int_0^\infty \sum_{i=1}^k \mathcal{G}(\tilde{ \bz}_1,\dots, \tilde{\bz}_{i-1},A_\pi \tilde{\bz}_i,\tilde{\bz}_{i+1},\dots, \tilde{\bz}_k) \\
&=\sum_{i=1}^k \mathcal{T}(\bz_1,\dots,\bz_{i-1}, A_\pi \bz_i,\bz_{i+1},\dots,\bz_k)
\end{align*}
which shows that $\mathcal{T}$ given in \eqref{eq:int_repr_lyap_eq} solves \eqref{eq:Lyapunov2} for $\bz_i \in\mD(A^2)$. By density  of $\mD(A^2)$ in $\mD(A)$, the equation remains valid for $\bz_i \in \mD(A)$ by continuity.

\emph{Uniqueness of $\mT$.} The uniqueness can be shown with the same arguments provided in \cite{BreKP19} and is therefore skipped at this point.
\end{proof}

\begin{remark}
 In the estimates of the previous proof it was used in \eqref{eq:lyap_aux2} and \eqref{eq:lyap_aux3} that only one coordinate belongs to $V'$ and the others are in $V$.
\end{remark}

With regard to a result similar to \eqref{eq:Lyapunov2} but for a different right hand side in \eqref{eq:Lyapunov2}, consider $\mathcal{F} \in S_{i+1}(V,V')$. For $\bz_1,\dots,\bz_i \in V$, the definition of $S_{i+1}(V,V')$ yields that $\mathcal{F}(\cdot,\bz_1,\dots,\bz_i) \in \mathcal{L}(V',\mathbb R)$. We particularly have that $\mathcal{F}(\cdot,\bz_1,\dots,\bz_i) \in \mathcal{L}(Y,\mathbb R)$. Identifying the last term with its Riesz representative in $Y$, we can define $B^* \mathcal{F}(\cdot,\bz_1,\dots,\bz_i)\in U$. With the same technique used in the proof Theorem \ref{thm:existenceUniquenessLyapunov_1}, we obtain the following result.

\begin{theorem}\label{thm:existenceUniquenessLyapunov_2}
Let $k\ge 3,i\in \{2,\dots,k-2\}, \mathcal{F}_1\in \mathcal{S}_{i+1}(V,V')$ and $\mathcal{F}_2 \in \mathcal{S}_{k-i+1}(V,V')$. Then, there exists a unique $\mathcal{T} \in S_k(V,V')$ such that for all $(\bz_1,\dots,\bz_k) \in \mD(A)^k$,
\begin{align}\label{eq:Lyapunov3}
\sum_{i=1}^k \mT(\bz_1,\dots,\bz_{i-1},A_\pi \bz_i,\bz_{i+1},\dots,\bz_k ) =\mathcal{G}(\bz_1,\dots,\bz_k),
\end{align}
 where:
\begin{equation*}
\mathcal{G}(\bz_1,\dots,\bz_k)
=  \langle B^* \mathcal{F}_1 (\cdot,\bz_1,\dots,\bz_i), B^* \mathcal{F}_2 (\cdot,\bz_{i+1},\dots,\bz_k) \rangle_U.
\end{equation*}
\end{theorem}

For what follows,  let us briefly recall a symmetrization technique introduced in \cite{BreKP19}. Let $i$ and $j \in \mathbb{N}$, consider
\begin{equation*}
S_{i,j}= \big\{ \sigma \in S_{i+j} \,|\, \sigma(1)< \dots < \sigma(i) \text{ and }
\sigma(i+1) < \dots < \sigma(i+j) \big \},
\end{equation*}
where $S_{i+j}$ is the set of permutations of $\{ 1,\dots,i+j \}$.
A permutation $\sigma \in S_{i,j}$ is uniquely defined by the subset
$\{\sigma(1),\dots,\sigma(i) \}$, therefore, the cardinality of $S_{i,j}$ is equal
to the number of subsets of cardinality $i$ of $\{ 1,\dots,i+j\}$ and, hence, $
|S_{i,j}|= \begin{pmatrix} i+j \\ i \end{pmatrix}.$
For a multilinear form $\mathcal{T}$ of order $i+j$, we set
\begin{equation} \label{eq:defSym}
\text{Sym}_{i,j}(\mathcal{T})(\bz_1,\dots,\bz_{i+j}) = \begin{pmatrix} i+j \\ i
\end{pmatrix}^{-1} \Big[ \sum_{\sigma \in S_{i,j}} \mathcal{T}
(\bz_{\sigma(1)},\dots,\bz_{\sigma(i+j)} ) \Big].
\end{equation}

\begin{theorem} \label{thm:mult_lin_form}
There exists a unique sequence of symmetric multilinear forms $(\mT_k)_{k \geq 2,}$ with $\mT_k \in
\mathcal{S}_k(V,V')$ and a unique sequence of multilinear forms $(\mR_{k})_{k \geq 3},$ with $\mR_{k} \in
\mM(\mD(A)^k,\R)$ such that for all $(\bz_1,\bz_2) \in V^2$,
\begin{equation} \label{eqLyapounov3a}
\mathcal{T}_2(\bz_1,\bz_2):= \langle \bz_1, \Pi \bz_2 \rangle_Y
\end{equation}
and such that for  all $k \geq 3$, for all $(\bz_1,...,\bz_k) \in \mathcal{D}(A)^k$,
\begin{equation}\label{eqLyapunov3b}
\sum_{i=1}^k \mathcal{T}_k (\bz_1,...,\bz_{i-1},A_\pi \bz_i, \bz_{i+1},...,\bz_k)=
\mathcal{R}_{k}(\bz_1,...,\bz_k),
\end{equation}
where
\begin{equation}\label{eq:aux1}
\begin{aligned}
& \mathcal{R}_{k}(\bz_1,\dots,\bz_k)= \frac{1}{2\alpha}\sum_{i=2}^{k-2} \begin{pmatrix} k \\ i \end{pmatrix}
\Symm\limits_{i,k-i} \big( \mathcal{C}_i \otimes
\mathcal{C}_{k-i} \big)(\bz_1,\dots,\bz_k) \\
&\hspace{1.5cm}
+ \sum_{j=1}^{k-1} \sum_{i=1}^{k-j} \mathcal{T}_{k-1} (A_0(\bz_j,\bz_{j+i}),\bz_1,\dots,\bz_{j-1},\bz_{j+1},\dots,\bz_{j+i-1},\bz_{j+i+1},\dots,\bz_k)
\end{aligned}
\end{equation}
with $
 \mathcal{C}_{i}(\bz_1,\dots,\bz_i) =  \displaystyle{B^* \mathcal{T}_{i+1}(\cdot,\bz_1,\dots,\bz_i)}$, and $\sum_{i=2}^r =0$ for $r<2$.
\end{theorem}
\begin{proof}
The statement follows by induction over $k$. We begin with $k=2$. By definition and well-known results for linear quadratic control problems, see, e.g., \cite{CurZ95,LasT00}, we obtain that $\mT_2 \in \mathcal{M}(Y\times Y,\mathbb R)$. Moreover, the operator $\Pi$ is the unique stabilizing solution to the algebraic Riccati equation. Hence, $\mathcal{T}_2$ is unique. Let us show that $\mT_2 \in S_2(V,V')$. For this purpose, note that the Riccati equation \eqref{eq:are} can be rewritten as a Lyapunov equation for the closed-loop system:
\begin{align*}
 \langle \bz_{2},A_\pi^{*} \bz_{1 } \rangle_{Y} + \langle  A_\pi\bz_{1},\bz_{2}\rangle _{Y}+\langle \bz_{2},\bz_{1} \rangle_{Y} + \frac{1}{\alpha} \langle B^{*} \Pi \bz_{1},B^{*}\Pi \bz_{2} \rangle_{U} = 0, \ \ \forall \bz_1,\bz_2 \in \mD(A).
\end{align*}
Similar as in the proof of Theorem \ref{thm:existenceUniquenessLyapunov_1}, we have the explicit integral representation
\begin{align*}
 \mT_2(\bz_1,\bz_2) = -\int_0^\infty \langle \tilde{\bz}_{2},\tilde{\bz}_{1} \rangle_{Y} + \frac{1}{\alpha} \langle B^{*} \Pi \tilde{\bz}_{1},B^{*}\Pi \tilde{\bz}_{2} \rangle_{U} \, \mathrm{d}t,
\end{align*}
where $\tilde{\bz}_i(t) = e^{A_\pi t} \bz_i,i=1,2$. This implies the estimate
\begin{align*}
 | \mT_2(\bz_1,\bz_2)| & \le M \int_0^\infty \| \tilde{\bz}_1\| _Y \| \tilde{\bz}_2\|_Y \, \mathrm{d}t \le M \| \tilde{\bz}_1\|_{L^2(0,\infty;Y)} \| \tilde{\bz}_2\|_{L^2(0,\infty;Y)} \le M \| \bz_1\|_{V'} \|\bz_2\|_{V'}.
\end{align*}
We thus have $\mT_2\in \mM(V'\times V',\mathbb R)$ which in particular implies $\mT_2\in S_2(V,V')$.

For $k\ge 3$, the equations \eqref{eqLyapunov3b} are linear and existence and uniqueness of $\mT_k \in S_k(V,V')$ and $\mR \in \mathcal{M}(\mD(A)^k,k)$ follow from Theorem \ref{thm:existenceUniquenessLyapunov_1} and Theorem \ref{thm:existenceUniquenessLyapunov_2}. Symmetry follows from the explicit integral representation of $\mathcal{T}_k$ as well as from the symmetry of $\mathcal{R}_k$ which is a consequence of the relation
\begin{align*}
&\sum_{j=1}^{k-1} \sum_{i=1}^{k-j} \mT_{k-1} (A_0(\bz_j,\bz_{j+i}),\bz_1,\dots,\bz_{j-1},\bz_{j+1},\dots,\bz_{j+i-1},\bz_{j+i+1},\dots,\bz_k) \\
&\quad=\sum_{j=1}^{k-1} \sum_{i=1}^{k-j} \langle \mT_{k-1}(\cdot, \bz_1,\dots,\bz_{j-1},\bz_{j+1},\dots,\bz_{j+i-1},\bz_{j+i+1},\dots,\bz_k), A_0(\bz_j,\bz_{j+i})\rangle_Y \\
&\quad=\frac{k(k-1)}{2} \Symm_{k-2,2} \left( \mT_{k-1} \otimes A_0 \right)(\bz_1,\dots,\bz_k).
\end{align*}
\end{proof}

\begin{remark}\label{rem:necessity}
We point out that it is essential to allow that $\mT_k \in S(V,V')$ rather than  $\mM(V^k,\R)$. In fact, since $B^{*}\in \mathcal{L}(Y,U)$ the first summands on the right hand side of \eqref{eq:aux1} would otherwise not be well-defined, though it would suffice to demand $\mT_k \in S(V,Y)$. Similarly the second summands would not be well-defined, since they contain the terms $A_0(\bz_i, \bz_{j+i})$.
\end{remark}

\begin{remark}\label{rem:symm_lyap}
	For the analysis of the polynomial feedback law below it will be convenient to note that for the special case $\bz_{i}=\by\in \mD(A)$ for $i=1,\dots,k$ with $k\ge 3$, we obtain
	\begin{align*}
		k\mT_{k}(A\by,\by,\dots,\by)&= \frac{1}{2\alpha} \sum_{i=1}^{k-1} \begin{pmatrix} k \\ i \end{pmatrix} \langle B^{*} \mT_{i+1}(\cdot,\by,\dots,\by),B^{*} \mT_{k-i+1}(\cdot,\by,\dots,\by) \rangle_{U} \\
		&\qquad + k(k-1) \mT_{k-1}(F(\by),\by,\dots,\by).
	\end{align*}
\end{remark}

\section{Polynomial feedback control}

Let us next analyze the polynomial feedback law
 $u_d \colon V \rightarrow U$ obtained by
\begin{align}\label{eq:pol_feedback}
u_{d}(\by_d) =  - \frac{1}{\alpha} \sum_{k=2}^{d} \frac{1}{(k-1)!} B^{*} \mT_k(\cdot,\by_d,\dots,\by_d),
\end{align}
with $\mT_k$ given in \eqref{eqLyapounov3a} and \eqref{eqLyapunov3b}. We then obtain the following closed-loop system
\begin{equation} \label{eq:cls}
\dot{\by}_d = A \by_d - F(\by_d) + Bu_{d}(\by_d), \quad \by_d(0) = \by_0.
\end{equation}

The subsequent proofs rely on local Lipschitz continuity estimates for the nonlinear part of the feedback law. It will therefore be convenient to introduce
\begin{align}\label{eq:nonl_feed}
 G_k(\by):=-\frac{1}{\alpha (k-1)!}BB^{*}\mT_k(\cdot,\by,\dots,\by),
\end{align}
for each $k \geq 3$ such that we have
\begin{equation}\label{eq:loc_stab_cls}
\begin{aligned}
\dot{\by}_d
 &=A_\pi \by_d - F(\by_d) -\frac{1}{\alpha} \sum_{k=3}^d \frac{1}{(k-1)!}BB^{*}\mT_k(\cdot,\by_d,\dots,\by_d) \\
 &= A_{\pi} \by_d - F(\by_d) + \sum_{k=3}^{d} G_{k}(\by_d).
\end{aligned}
\end{equation}

As mentioned before, for $\mT_k\in S_k(V,V')$ and for fixed $\by \in V$, the term $\mT_k(\cdot,\by,\dots,\by) \in \mL(Y,\mathbb R)$. With $g_k(\by,\dots,\by) \in Y$ let us denote its Riesz representative. Hence, we obtain
\begin{align*}
\|g_k(\by,\dots,\by)\|_Y = \sup_{\bz \in B_Y(1)} |\mT_k(\bz,\by,\dots,\by)| \le M   \|\by\|^{k-1}_V.
  \end{align*}
One can easily show the following local Lipschitz estimate for $G_k$ which extends the 2-D result given in \cite{BreKP19b}.

 \begin{lemma}\label{lem:Lipschitz_G}
 For all $k\ge 3$, there exists a constant $C(k) >0$ such that for all $\by$ and $\bz \in V$,
 \begin{align*}
 \| G_k(\by)-G_k(\bz) \| _ Y \le C(k) \| \by - \bz \| _V \max(\| \by \| _V , \| \bz \| _V)^{k-2}.
 \end{align*}
 Moreover, for all $\delta \in [0,1]$, for all $\tilde{\by}$ and $\tilde{\bz} \in W_\infty(\mD(A_\lambda),Y)$ such that    $\|\tilde{\by}\|_{W_\infty(\mD(A_\lambda),Y)}\le \delta$ and $\|\tilde{\bz}\|_{W_\infty}(\mD(A_\lambda),Y)\le \delta$,
  \begin{align*}
   \|G_k(\tilde{\by})-G_k(\tilde{\bz})\|_{L^2(0,\infty;Y)} \le C(k)\delta\| \tilde{\by}-\tilde{\bz}\|_{W_\infty(\mD(A_\lambda),Y)}.
  \end{align*}
\end{lemma}

As a consequence, we obtain the local well-posedness of the closed-loop system.
\begin{theorem} \label{thm:cls_well_posed}
  Let $C$ and $C(k)$ denote the constants from Corollary \ref{cor:F_Lip} and Lemma \ref{lem:Lipschitz_G}. There exists a constant $M_{\mathrm{cls}}>0$ such that for all $\by_0 \in V$ with
  \begin{align*}
   \| \by _0\| _V   \le \frac{1}{4(C+\sum_{k=3}^dC(k)M_{\mathrm{cls}}^2)}
  \end{align*}
  the closed-loop system \eqref{eq:cls} has a unique solution $\by_d \in W_\infty(\mD(A_\lambda),Y)$, which satisfies
  \begin{equation} \label{eq:estimate_cls}
\| \by_d\|_{W_\infty(\mD(A_\lambda),Y)} \leq 2 M_{\mathrm{cls}} \| \by_0 \|_V.
\end{equation}
\end{theorem}

\begin{proof}
Similar to the proof of Lemma \ref{lem:non_loc_sol}, we obtain the existence of a solution
 $\by \in W_\infty(\mD(A_\lambda),Y)$, satisfying \eqref{eq:estimate_cls}.
Let us therefore focus on uniqueness and denote by $\by$ and $\bz$ two solutions to \eqref{eq:cls} in $W_\infty(\mD(A_\lambda),Y)$. We set $\mathbf{e}= \by - \bz$. Again, as in the proof of Lemma \ref{lem:non_loc_sol}, there exists  $M>0$ such that
\begin{equation*}
\frac{1}{2} \frac{\text{d}}{\text{d} t} \| \mathbf{e}(t) \|_Y^2
\leq M \Big( 1 + \| \by(t) \|_{\bbH^2(\Omega)}^2+ \| \bz(t) \|_{\bbH^2(\Omega)}^2 +
\sum_{k=3}^d C(k)^2 \max ( \| \by(t) \|_V, \| \bz(t) \|_V )^{2(k-2)} \Big) \| \mathbf{e}(t) \|_Y^2,
\end{equation*}
for all $t \geq 0$.
Since $\by$ and $\bz \in W_\infty(\mD(A_\lambda),Y)$ and $\mathbf{e}(0)=0$, we obtain with Gronwall's inequality that $\mathbf{e}=0$.
\end{proof}

\section{Error estimates}

In this section, we analyze the feedback law \eqref{eq:pol_feedback} and compare it to the optimal value  $\mathcal{V}(\by_0)$. We follow a strategy used in \cite{BreKP19} which is based on a polynomial function $\mV_d$ of the form
\begin{align}\label{eq:poly_val_func}
\mV_d\colon V \to \mathbb R, \ \ \mV_d(\by):=\sum_{k=2}^d \frac{1}{k!} \mT_k(\by,\dots,\by).
\end{align}
The motivation for the specific definition of $\mV_d$ is that in the 2-D case, the sequence of multilinear forms $\mT_k$ coincides with derivatives $D^k\mV(0)$ of the value function considered as continuous multilinear forms on $Y^k$. Hence, in that case the expression for $\mV_d$ represents a Taylor series expansion of $\mV$ around 0 in the $Y$ topology. In the 3-D case we utilize the structure of the 2-D case to propose approximating feedback controls based on the generalized Lyapunov equations from Theorem \ref{thm:mult_lin_form}. Using that $\mathcal{V}$ is differentiable on $V$, we can eventually derive error estimates analogous  to those obtained in \cite{BreKP19b}.



We begin by showing that $\mV_d$ as defined in \eqref{eq:poly_val_func} satisfies a perturbed HJB equation. For this purpose, consider the following polynomial term for $\by \in V$:
\begin{equation} \label{eq:poly_val_func}
\begin{aligned}
r_d(\by)&:= \frac{1}{2\alpha}\sum_{k=d+1}^{2d-2} \sum_{\ell=k-d+1}^{d-1}\frac{1}{\ell!(k-\ell)!}  \left\langle  B^{*} \mT_{\ell+1}(\cdot,\by,\dots,\by) ,B^{*} \mT_{k-\ell+1}(\cdot,\by,\dots,\by)  \right\rangle_{U}.
 \\ &\qquad
 + \frac{1}{(d-1)!}\mT_{d}(F(\by),\by,\dots,\by).
\end{aligned}
\end{equation}
Before stating the announced result, note that $\mV_d$ is Fr\'{e}chet differentiable on $V$ with
\begin{align}\label{eq:poly_val_der}
D\mV_d(\by)  = \sum_{k=2}^d \frac{1}{(k-1)!} \mT_k (\cdot,\by,\dots,\by).
\end{align}
Moreover, by Theorem \ref{thm:mult_lin_form}, we know that $D\mV_d$ can be uniquely extended to an element in $\mathcal{L}(Y,\mathbb R)$. As an element in $\mathcal{L}(Y,\mathbb R)$ it satisfies the announced perturbed HJB equation.

\begin{proposition}\label{prop:resi_hjb}
For all $d \geq 2$ and all $\by \in \mathcal{D}(A)$, we have
\begin{equation} \label{eq:resi_hjb1}
	\begin{aligned}
  D\mathcal{V}_d(\by)(A\by-F(\by)) +
\frac{1}{2} \| \by \|_Y^2 - \frac{1}{2\alpha }
\langle B^{*}D\mathcal{V}_d(\by),B^{*}D\mathcal{V}_d(\by)
\rangle_{U}+r_d(\by)=0 .
\end{aligned}
\end{equation}
Moreover, for all $d \geq 2$, there exists a constant $C>0$ such that for all $\by
\in V$,
\begin{equation} \label{eq:resi_hjb2}
|r_d(\by)| \leq C \sum_{i=d+1}^{2d}  \| \by \|_{V}^{i}.
\end{equation}
\end{proposition}

\begin{proof}
Let us prove \eqref{eq:resi_hjb1}. We fix $\by \in \mathcal{D}(A)$. Since  $\mathcal{T}_2$ is characterized by $\Pi$ which satisfies the algebraic Riccati equation \ref{eq:are}, we obtain for $d=2$
\begin{align*}
& -  D\mathcal{V}_d(\by)(A\by-F(\by)) -
\frac{1}{2} \| \by \|_Y^2 + \frac{1}{2\alpha } \langle B^{*} D\mV_{d}(\by) ,B^{*} D\mV_{d}(\by)\rangle_{U} \\
& \qquad = -\mathcal{T}_2(A\by-F(\by),\by)  - \frac{1}{2}\| \by \|_Y^2  + \frac{1}{2\alpha } \langle B^{*} \mT_{2}(\cdot,\by),B^{*} \mT_{2}(\cdot,\by)\rangle_{U} \\
& \qquad = \underbrace{-\mathcal{T}_2(A\by,\by) - \frac{1}{2} \| \by \|_Y^2 + \frac{1}{2\alpha}  \langle B^{*} \mT_{2}(\cdot,\by),B^{*} \mT_{2}(\cdot,\by)\rangle_{U} }_{= 0} + \mT_{2}(F(\by),\by) =  r_2(\by).
\end{align*}

Now let $d\geq 3$. Our proof is based on Theorem \ref{thm:mult_lin_form}. From Remark \ref{rem:symm_lyap}, we know that the expressions of the multilinear forms can be simplified when the mappings are evaluated at $(\by,\dots,\by)\in Y^{i}$ and $(\by,\dots,\by)\in Y^{k}$, respectively. In particular, we have
\begin{equation} \label{eq:rpProof3}
\begin{aligned}
k \mathcal{T}_k(A\by, \by,\dots,\by)&= \frac{1}{2\alpha} \sum_{i=1}^{k-1} \begin{pmatrix} k \\ i \end{pmatrix}\langle B^{*} \mT_{i+1}(\cdot,\by,\dots,\by),B^{*}\mT_{k-i+1}(\cdot,\by,\dots,\by ) \rangle_{U}\\
&\qquad+k(k-1)\mT_{k-1}(F(\by),\by,\dots,\by )
\end{aligned}
\end{equation}
We are now ready to prove \eqref{eq:resi_hjb1}.
By \eqref{eq:poly_val_der}  we have
\begin{equation} \label{eqRPExpansion1}
D\mathcal{V}_d(\by) (A\by-F(\by)) = \sum_{k=2}^d \frac{1}{(k-1)!} \mathcal{T}_k \big( A\by, \by,\dots,\by \big)-\sum_{k=3}^{d+1} \frac{1}{(k-2)!}\mathcal{T}_{k-1} \big( F(\by), \by,\dots,\by \big),
\end{equation}
and in a similar manner
\begin{align*}
B^{*}D\mathcal{V}_d(\by) = & \ \sum_{i=2}^d \frac{1}{(i-1)!} B^{*} \mathcal{T}_i(\cdot, \by,\dots,\by) =\sum_{i=1}^{d-1} \frac{1}{i!} B^{*} \mathcal{T}_{i+1}(\cdot, \by,\dots,\by).
\end{align*}
As a consequence, we obtain
\begin{equation} \label{eqRPExpansion0}
\begin{aligned}
&\langle B^{*} D\mathcal{V}_d(\by),B^{*} D\mathcal{V}_d(\by) \rangle_{U} =
\left\langle \sum_{i=1}^{d-1} \frac{1}{i!} B^{*} \mT_{i+1}(\cdot,\by,\dots,\by),\sum_{j=1}^{d-1} \frac{1}{j!} B^{*} \mT_{j+1}(\cdot,\by,\dots,\by) \right\rangle_{U} \\
&\qquad = \langle B^{*} \mT_{2}(\cdot,\by),B^{*}\mT_{2}(\cdot,\by) \rangle_{U} + \sum_{k=3}^{d} \sum_{\ell=1}^{k-1}\frac{1}{\ell! (   k-\ell)!} \left \langle  B^{*} \mT_{\ell+1}(\cdot,\by,\dots,\by), B^{*} \mT_{ k-\ell+1}(\cdot,\by,\dots,\by) \right\rangle_{U} \\
&\qquad \qquad + \sum_{k=d+1}^{2d-2}\sum_{\ell=k-d+1}^{d-1}\frac{1}{\ell! (   k-\ell)!} \left \langle  B^{*} \mT_{\ell+1}(\cdot,\by,\dots,\by),  B^{*} \mT_{ k-\ell+1}(\cdot,\by,\dots,\by) \right\rangle_{U}.
\end{aligned}
\end{equation}
From \eqref{eqRPExpansion1} and \eqref{eqRPExpansion0}, we conclude that
\begin{align*}
& - D\mathcal{V}_d(\by)(A\by-F(\by)) -
\frac{1}{2} \| \by \|_Y^2 + \frac{1}{2\alpha }
\left \langle B^{*}D\mV_{d}(\by),B^{*}D\mV_{d}(\by) \right \rangle_{U} \\
& \quad =  -\frac{1}{2} \left[ 2 \mathcal{T}_2(A_\pi \by, \by) + \| \by \|_Y^2 - \frac{1}{\alpha} \langle B^{*}\mT_2(\cdot,\by),B^{*}\mT_2(\cdot,\by) \rangle_{U} \right]\\
& \qquad   - \sum_{k=3}^d \frac{1}{k!} \Bigg[ k \mathcal{T}_k(A \by, \by,\dots,\by) -  k(k-1)\mT_{k-1}(F(\by),\by,\dots,\by) \\
&\qquad \qquad \qquad\left.-\frac{1}{2\alpha}\sum_{\ell=1}^{k-1} \begin{pmatrix} k \\ \ell \end{pmatrix}\left \langle   B^{*} \mT_{\ell+1}(\cdot,\by,\dots,\by),  B^{*} \mT_{ k-\ell+1}(\cdot,\by,\dots,\by) \right\rangle_{U}\right] \\
&\qquad   + \sum_{k=d+1}^{2d-2}\sum_{\ell=k-d+1}^{d-1} \frac{1}{\ell!(   k-\ell)!}  \left \langle B^{*} \mT_{\ell+1}(\cdot,\by,\dots,\by), B^{*} \mT_{ k-\ell+1}(\cdot,\by,\dots,\by) \right\rangle_{U} \\
&\qquad +\frac{1}{(d-1)!} \mT_{d}(F(\by),\by,\dots,\by).
\end{align*}
The terms in brackets in the above expression are equal to zero by \eqref{eq:are} and \eqref{eq:rpProof3}.
This proves \eqref{eq:resi_hjb1}. For the estimate \eqref{eq:resi_hjb2}, we use $\mT_k \in S_k(V,V')$ and the definition \eqref{eq:poly_val_func} to obtain
\begin{align*}
| r_d(\by)| &\le M \sum_{k=d+1}^{2d-2} \sum_{\ell =k-d+1}^{d-1} \frac{1}{\ell! (k-\ell)!} \| \mT_{\ell+1}(\cdot,\by,\dots,\by)\|_Y \| \mT_{k-\ell+1}(\cdot,\by,\dots,\by)\|_Y \\
&\quad +M | \mT_d (F(\by),\by,\dots,\by) | \\
& \le M\left(\sum_{k=d+1}^{2d-2} \sum_{\ell =k-d+1}^{d-1}  \|\by\|_V^k   + \| F(\by)\|_{V'} \| \by\|_V^{d-1} \right).
\end{align*}
The assertion now follows with Proposition \ref{prop:N_estimates_2}.
\end{proof}

\begin{lemma} \label{lemma:estimateRp}
Let $d \geq 2.$ Then, there exists $\delta>0$ and a constant $M>0$ such that for
all $\by_0 \in B_V(\delta_0)$,
\begin{equation*}
\int_0^\infty r_d ( \bar{\by}(t))\, \mathrm{d} t \leq C \| \by_0 \|_{V}^{d+1} \quad \text{and} \quad \int_0^\infty r_d (\by_d(t))\, \mathrm{d} t \leq C \| \by_0\|_{V}^{d+1},
\end{equation*}
where $\bar{\by}$ is the optimal trajectory for problem \eqref{eq:NLQprob} with
initial value $\by_0$.
\end{lemma}

\begin{proof}
By Proposition \ref{proposition:UisOptimal} and Theorem \ref{thm:cls_well_posed}, for $\delta>0$ sufficiently small there exists a constant $C_1$ such
that for all $\by_0 \in B_V(\delta)$,
\begin{equation*}
\max\left(\| \by_d \|_{W_{\infty}(\mD(A_\lambda),Y)},\| \bar{\by} \|_{W_{\infty}(\mD(A_\lambda),Y)} \right) \leq C_1 \| \by_0 \|_V .
\end{equation*}
Since we can assume that $\| \by_0 \| _V \le 1$, the statement is a consequence of Proposition \ref{prop:resi_hjb}.
\end{proof}

Let us now consider a perturbation $J_d$ of the cost function $J$ of the form
\begin{equation*}
J_d(\by,u):= \frac{1}{2} \int_0^\infty \| \by \|_Y^2\, \dd t +
\frac{\alpha}{2} \int_0^\infty \|u\|_{U}^2 \, \dd t + \int_0^\infty r_d(\by)\, \dd t.
\end{equation*}

Next, we show that the polynomial feedback law $u_{d}(\by_d)=-\frac{1}{\alpha} B^{*}D\mV_{d}(\by_{d})$ with the corresponding trajectory $\by_d$ performs better than $(\bar{\by},\bar{u})$ with regard to the perturbed cost function $J_d$.

\begin{lemma} \label{lemma:optimalityJp}
Let $d \geq 2.$ Then there exists $\delta >0$ such that for all initial values $\by_0 \in B_V(\delta_0)$
\begin{equation*}
\mathcal{V}_d(\by_0)= J_d(\by_d,{u}_d) \leq J_d(\bar{\by},\bar{u}) \ \
\end{equation*}
where $(\bar{\by},\bar{u})$ is the optimal solution for problem \eqref{eq:NLQprob} with
initial value $\by_0$.
\end{lemma}

\begin{proof}
By Lemma \ref{lemma:estimateRp}, it follows that $J_d(\bar{\by},\bar{u})$ and $J_d(\by_d,u_d)$ are finite. We have that $\bar{\by} \in H^1(0,\infty;Y)$ and, hence, for all $T>0$, it holds that $\bar{\by}  \in W^{1,1}(0,T;Y)$. We can apply a chain rule established in \cite{BreKP19} to each of the bounded multilinear forms which
appear  in $\mathcal{V}_d({\bar{\by}}(\cdot))$.
Omitting the time variable in what follows, we obtain
\begin{align*}
\frac{\dd}{\dd t} \mathcal{V}_d({\bar{\by}})= D\mathcal{V}_d({\bar{\by}}) \big( A {\bar{\by}} -F({\bar{\by}}) +B{\bar{u}} \big) .
\end{align*}
By Proposition \ref{prop:resi_hjb},
\begin{align*}
 \frac{\dd}{\dd t} \mathcal{V}_d({\bar{\by}})&=-r_d({\bar{\by}}) -\frac{1}{2} \| {{\bar{\by}}} \|_Y^2 +
\frac{1}{2\alpha} \langle B^{*}D\mV_{d}({\bar{\by}}),B^{*}D\mV_{d}({\bar{\by}}) \rangle_{U} + D\mathcal{V}_d ({\bar{\by}})(B\bar{u}) \\
&= -\ell_d({\bar{\by}},\bar{u}) + \frac{1}{2\alpha} \langle B^{*}D\mV_{d}({\bar{\by}}),B^{*}D\mV_{d}({\bar{\by}}) \rangle_{U} + D\mathcal{V}_d ({\bar{\by}})(B\bar{u}) + \frac{\alpha}{2}\| \bar{u}\|_{U}^{2},
\end{align*}
where $\ell_d({\by},{u}):= \frac{1}{2} \| \by \|_Y^2 + \frac{\alpha}{2}\|{u}\|_{U}^2 + r_d({\by}) .$
Hence, it follows that
\begin{align} \label{eq:verifTh1}
 \frac{\dd}{\dd t} \mathcal{V}_d({\bar{\by}})&= - \ell_d({\bar{\by}},\bar{u}) +
\frac{\alpha}{2} \left\| \bar{u} + \frac{1}{\alpha} B^{*}D\mathcal{V}_d({\bar{\by}})\right\|_{U}^{2}.
\end{align}
We deduce that
\begin{equation} \label{eq:verifTh2a}
\mathcal{V}_d({\bar{\by}}(T)) - \mathcal{V}_d(\by_0) \geq - \int_0^T \ell_d({\bar{\by}},\bar{u}) \dd t.
\end{equation}
With a similar derivation for $u=u_d$, we infer that
\begin{equation} \label{eq:verifTh2b}
\mathcal{V}_d(\by_d(T)) - \mathcal{V}_d(\by_0) = - \int_0^T
\ell_d(\by_d,u_d) \dd t,
\end{equation}
since for this control, the squared expression vanishes.
We have $\lim_{T\to \infty} \bar{\by}(T)=0 \text{ and } \lim_{T\to \infty}\by_d(T)=0 \text{ in } V.$ Since $\mT_{k} \in S_{k}(V,V')$, this implies that
\begin{equation*}
\mathcal{V}_d({\bar{\by}}(T)) \overset{T \to \infty}{\longrightarrow 0} \quad
\text{and} \quad
\mathcal{V}_d(\by_d(T)) \underset{T \to \infty}{\longrightarrow} 0.
\end{equation*}
Finally, passing to the limit in \eqref{eq:verifTh2a} and \eqref{eq:verifTh2b}, we obtain
\begin{equation*}
J_d(\bar{\by},\bar{u}) = \int_0^\infty \ell_d(\bar{\by},\bar{u}) \geq \mathcal{V}_d(\by_0) = \int_0^\infty \ell_d(\by_d,u_{d}) = J_d(\by_d,u_d).
\end{equation*}
The lemma is proved.
\end{proof}

We now prove that $\mathcal{V}_d$ is a Taylor expansion of $\mathcal{V}$ and analyze the quality of the feedback law $u_d$ in the neighborhood of 0.

\begin{theorem} \label{thm:subOptimality}
There exists $\delta>0$ and a constant $M>0$ such that for all $\by_0\in B_V(\delta)$
\begin{align}
&\mV(\by_0)\le J(\by_d,u_d) \leq \mathcal{V}(\by_0) + 2M \|\by_0\|_{V}^{d+1}, \label{eq:subOptimality1} \\
& | \mathcal{V}(\by_0)-\mathcal{V}_d(\by_0) | \leq M \|\by_0\|_{V}^{d+1}.
\label{eq:subOptimality2}
\end{align}
\end{theorem}

\begin{proof}
The following inequalities follow directly from Lemma \ref{lemma:estimateRp} and Lemma
\ref{lemma:optimalityJp}:
\begin{align*}
\begin{array}{ll}
|\mathcal{V}_d(\by_0) - J(\by_d,u_{d})| \leq M \| \by_0 \|_{V}^{d+1}, \phantom{\Big|} \qquad &
\mathcal{V}_d(\by_0) \leq J_d(\bar{\by},\bar{u}), \\
|\mathcal{V}(\by_0) - J_d(\bar{\by},\bar{u})| \leq M \| \by_0 \|_{V}^{d+1},
\phantom{\Big|}&
\mathcal{V}(\by_0) \leq J(\by_d,u_d),
\end{array}
\end{align*}
where $\bar{u}$ is the unique solution to \eqref{eq:NLQprob} with initial value $\by_0$.
Therefore,
\begin{align*}
& J(\by_d,u_d)-2 M \| \by_0 \|_{V}^{d+1}
\leq \mathcal{V}_d(\by_0) - M \| \by_0 \|_{V}^{d+1}
\leq J_d(\bar{\by},\bar{u}) - M \| \by_0 \|_{V}^{d+1} \\
& \qquad \leq \mathcal{V}(\by_0) \leq J(\by_d,u_{d})
\leq \mathcal{V}_d(\by_0) + M \| \by_0 \|_{V}^{d+1},
\end{align*}
which proves inequalities \eqref{eq:subOptimality1} and
\eqref{eq:subOptimality2}.
\end{proof}

\begin{theorem}\label{theo19}
Let $d \geq 2$. There exist $\delta_6 > 0$ and $M > 0$ such that for all $\by_{0} \in B_V(\delta_6)$, it holds that
\begin{align*}
\| \bar{\by}-\by_{d} \|_{W_{\infty}(\mD(A_\lambda),Y)} & \le M \| \by _{0} \|_{V}^{d}, \\
\| \bar{u}-u_{d} \|_{L^{2}(0,\infty;U)},& \le  M \| \by _{0} \|_{V}^{d},
\end{align*}
where $(\bar{\by},\bar{u})= (\mathcal{Y}(\by_0),\mathcal{U}(\by_0))$, $\by_{d}$ is the solution of the closed-loop system \eqref{eq:cls} with initial condition $\by_0$, and $u_d$ is the associated control.
\end{theorem}
The proof  follows in part  the arguments provided in \cite{BreKP19} for the 2-D case. But  it also requires some changes  and thus it is given in the Appendix.

\section{Conclusions} An asymptotic expansion for the value function to an optimal control problem associated to the Navier-Stokes equation in dimension three was developed. The terms of the expansion are multilinear forms arising as the solutions to generalized Lyapunov equations. To achieve the desired approximation properties it is essential to consider the domains of these multilinear forms as an appropriate combination of $\bbH^1(\Omega)$ and $\bbL^2(\Omega)$ spaces. In future work the impact of the generalized Lyapunov equations for the numerical realization of feedback mechanisms is of interest. This will require an independent effort, however.


\appendix

\section{Proofs}\label{sec:appendix_proofs}

\begin{proof}[Proof of Lemma \ref{lem:non_loc_sol}]
 From \cite[Theorem II.1.3.1]{Benetal07} it follows that for all $(\by_0,\bg)\in V\times L^2(0,\infty;Y)$ the system
 \begin{align*}
  \dot{\bz}=A_s \bz + \bg , \quad \bz(0)=\by_0
 \end{align*}
has a unique solution $\bz \in W_\infty(\mD(A_s),Y)$. Additionally, there exists a constant $M_s$ such that
\begin{align}\label{eq:reg_est_stabsys}
  \| \bz \|_{W_\infty(\mD(A_s),Y)} \le M_s (\| \by _0 \| _V + \| \bg \|_{L^2(0,\infty;Y)} ).
\end{align}
Let us without loss of generality assume that $M_s \ge \frac{1}{2C}$.
  Similar to \cite[Lemma 5]{BreKP19b}, we are going to apply a fixed-point argument to the system \eqref{eq:non_loc_sol}.
  For this purpose, let us define
  $$\mathcal{M}=\left\{ \by \in W_\infty(\mD(A_s),Y) \ | \ \|\by\|_{W_\infty(\mD(A_s),Y)} \le 2 M_s \gamma \right\}$$
  as well as the mapping $\mathcal{Z}\colon \mathcal{M} \ni \by \mapsto \bz=\mathcal{Z}(\by)\in W_\infty(\mD(A_s),Y),$ where $\bz$ is the unique solution of
\begin{align*}
  \dot{\bz}=A_s\bz - F(\by) + \bff , \quad \bz(0)=\by_0.
\end{align*}
If there exists a fixed point of $\mathcal{Z}$, then it is a unique solution of \eqref{eq:non_loc_sol} in $\mathcal{M}$.
With $C$ and $M_s$ given, we shall use Corollary \ref{cor:F_Lip} with $\delta=2M_s\gamma\le\frac{1}{2CM_s}\le 1$ and \eqref{eq:reg_est_stabsys} to obtain
\begin{align*}
  \| \bz \| _{W_\infty(\mD(A_s),Y)} &\le M_s ( \| F(\by) \|_{L^2(0,\infty;Y)} + \| \bff \|_{L^2(0,\infty;Y)} + \| \by_0 \| _V )  \\
 & \le M_s \left(\frac{1}{2M_s} \| \by \|_{W_\infty(\mD(A_s),Y)}+ \gamma \right) \le 2  M_s  \gamma.
\end{align*}
This implies $\mathcal{Z}(\mathcal{M})\subseteq \mathcal{M}$. For $\by_1,\by_2\in \mathcal{M}$ consider $\bz=\mathcal{Z}(\by_1)-\mathcal{Z}(\by_2)$ solving
\begin{align*}
  \dot{\bz}= A_s\bz - F(\by_1)+F(\by_2), \quad \bz(0)=0.
\end{align*}
Again by \eqref{eq:reg_est_stabsys} and Corollary \ref{cor:F_Lip}, it follows that
\begin{align*}
\| \mathcal{Z}(\by_1)-\mathcal{Z}(\by_2) \| _{W_\infty(\mD(A_s),Y)} &=  \| \bz \| _{W_\infty(\mD(A_s),Y)}\le M_s (\| F(\by_1)-F(\by_2) \|_{L^2(0,\infty;Y)} ) \\
 &\le M_s \delta C \| \by_1 - \by_2 \| _{W_\infty(\mD(A_s),Y)} \le \frac{1}{2} \| \by_1 -\by_2 \| _{W_\infty(\mD(A_s),Y)}.
\end{align*}
Hence, $\mathcal{Z}$ is a contraction in $\mathcal{M}$ and there exists a unique $\by \in \mathcal{M}$ with $\mathcal{Z}(\by)=\by$. Regarding uniqueness in $W_{\infty}(\mD(A_s),Y)$, consider two solutions $\by,\bz \in W_{\infty}(\mD(A_s),Y)$. The difference $\be :=\by-\bz$ then satisfies
\begin{align*}
	\dot{\be}=A_{s} \be  - F(\by)+F(\bz), \ \ \be(0)=0.
\end{align*}
Multiplication with $\be$ and subsequent integration yields
\begin{align*}
	\frac{1}{2}\frac{\mathrm{d}}{\mathrm{d}t}\| \be \|_{Y}^{2} = \langle A_{s} \be ,\be \rangle_{Y} - \langle F(\by)-F(\bz),\be \rangle_{V',V}.
\end{align*}
Note that $A_{s}$ satisfies an expression analogous to \eqref{eq:VY-coerc}. We thus have
\begin{align*}
	\frac{1}{2}\frac{\mathrm{d}}{\mathrm{d}t} \| \be \| _{Y}^{2} \le \alpha \| \be \|_{Y}^{2} - \beta \| \be \|_{V}^{2} + \| F(\by)-F(\bz)\|_{V'} \| \be \|_{V},
\end{align*}
with $\alpha \ge 0$ and $\beta >0 $.
From Proposition \ref{prop:N_estimates} and Young's inequality we conclude that
\begin{align*}
	\frac{1}{2}\frac{\mathrm{d}}{\mathrm{d}t} \| \be \| _{Y}^{2} &\le \alpha \| \be \|_{Y}^{2} - \beta \| \be \|_{V}^{2} + M\left( \| \be \|_{Y}\| \by \|_{\mathbb{H}^2(\Omega)}+ \| \be \|_{Y} \| \bz \|_{\mathbb{H}^2(\Omega)}  \right) \| \be \|_{V}\\
	&\le \alpha \| \be \|_{Y}^{2} - \beta \| \be \|_{V}^{2} + \frac{M}{\iota}\| \be \|^{2}_{V}  + \frac{M\iota}{2}\| \be \|_{Y}^2 \left(  \| \by \|_{\mathbb{H}^2(\Omega)}^2 + \| \bz \|_{\mathbb{H}^2(\Omega)}^2 \right) .
\end{align*}
Choosing $\iota$ large enough, this yields	
\begin{align*}
	\frac{1}{2}\frac{\mathrm{d}}{\mathrm{d}t} \| \be \| _{Y}^{2}
	&\le  \left( \alpha+ \frac{M\iota}{2} (\| \by \|^2_{\mathbb{H}^2(\Omega)}+\| \bz \|^2_{\mathbb{H}^2(\Omega)} ) \right)\| \be \| _{Y} ^2.
\end{align*}
Since $\by,\bz \in W_{\infty}(\mD(A_s),Y)$ and $\be(0)=0$ we can apply Gronwall's inequality and obtain that $\be(t)=0$ for all $t\ge 0$. This shows uniqueness of solutions in $W_{\infty}(\mD(A_s),Y)$.
\end{proof}

\begin{proof}[Proof of Lemma \ref{lemma:costateEqualDerivative}]
By continuity of the mapping $\mathcal{Y}$, there exists $\delta_5 \in (0,\delta_4]$ such that for all $\by_0 \in B_V(\delta_5)$, $\| \mathcal{Y}(\by_0) \|_{L^\infty(0,\infty;V)} \leq \delta_4$.

We now claim the following: for all $\by_0 \in B_V(\delta_5)$, we have $\bp(0)= D\mathcal{V}(\by_0)$, where $\bp= \mathcal{P}(\by_0)$. To verify this claim, let $\by_0$ and $\tilde{\by}_0 \in B_V(\delta_5),$ and set $(\by,u,\bp)= (\mathcal{Y}(\by_0),\mathcal{U}(\by_0),\mathcal{P}(\by_0))$ and $(\tilde{\by},\tilde{u})= (\mathcal{Y}(\tilde{\by}_0),\mathcal{U}(\tilde{\by}_0))$.  We have
\begin{align}
\mathcal{V}(\tilde{\by}_0)- \mathcal{V}({\by}_0)
= \ & \Big( \frac{1}{2} \| \tilde{\by} \|^2 + \frac{\alpha}{2} \| \tilde{u} \|^2 \Big) -
\Big( \frac{1}{2} \| {\by} \|^2 + \frac{\alpha}{2} \| {u} \|^2 \Big) \notag \\
& \qquad - \big\langle \bp, \dot{\tilde{\by}}-(A\tilde{\by}-F(\tilde{\by})+B\tilde{u}) \big\rangle_{L^2(0,\infty;Y)} \notag \\
& \qquad + \big\langle \bp, \dot{{\by}}- (A{\by} - F(\by)+B{u}) \big\rangle_{L^2(0,\infty;Y)}. \label{eq:diffValFunc}
\end{align}
Indeed, $u$ and $\tilde{u}$ are optimal and the last two terms vanish.
The following four relations can be easily verified:
\begin{equation} \label{eq:diffValFunc2}
\begin{aligned}
\frac{1}{2} \| \tilde{\by} \|^2 - \frac{1}{2} \|  {\by} \|^2
= \ & \langle {\by}, \tilde{\by}- {\by} \rangle_{L^2(0,\infty;Y)} + \frac{1}{2} \| \tilde{\by}-{\by}\|^2, \\
\frac{\alpha}{2} \| \tilde{u} \|^2 - \frac{\alpha}{2} \| {u} \|^2
= \ & \alpha \langle {u}, \tilde{u}- {u} \rangle_{L^2(0,\infty;U)} + \frac{\alpha}{2} \| \tilde{u} - {u} \|^2, \\
F(\tilde{\by})-F(\by)
= \ & F(\tilde{\by}-\by)+N(\tilde{\by}-\by,\by)+N(\by,\tilde{\by}-\by), \\
= \ & ((\tilde{\by}-\by)\cdot \nabla) (\tilde{\by}-\by) + ((\tilde{\by}-\by) \cdot \nabla) \by + (\by \cdot \nabla) (\tilde{\by}-\by) \\
-\langle \bp,\dot{\tilde{\by}} -\dot{\by} \rangle_{L^2(0,\infty;Y)}
= \ & \langle \bp(0),\tilde{\by}_0-{\by}(0) \rangle_{V',V} + \langle \dot{\bp}, \tilde{\by}- {\by} \rangle_{L^2(0,\infty;[\mD(A_\lambda)]'),L^2(0,\infty;\mD(A_\lambda))}.
\end{aligned}
\end{equation}
Combining \eqref{eq:diffValFunc} and \eqref{eq:diffValFunc2} yields
\begin{align*}
\mathcal{V}(\tilde{\by}_0)-\mathcal{V}(\by_0)
= \ & \langle \bp(0), \tilde{\by}(0)-{\by}(0) \rangle_{V',V} + \frac{1}{2} \| \tilde{\by}-{\by} \|_{Y}^2
+ \frac{\alpha}{2} \| \tilde{u}- {u} \|_{U}^2 \\
& \quad - \big\langle \bp, F(\tilde{\by}-{\by}) \big\rangle_{L^2(0,\infty;Y)}   + \big\langle \underbrace{\alpha {u} + B^{*}\bp}_{=0}, \tilde{u}- {u} \big\rangle_{L^2(0,\infty;U)} \\
& \quad + \big\langle   \underbrace{\dot{\bp} + A'\bp + \by +P((\by\cdot \nabla)\bp-(\nabla \by)^T \bp)}_{=0} , \tilde{\by}- {\by} \big\rangle_{L^2(0,\infty;[\mD(A_\lambda)]'    ,L^2(0,\infty;\mD(A_\lambda)))} .
\end{align*}
For $\tilde{\by}_0= {\by}_0 + \mathbf{h}$, we have $\| \tilde{\by} - \by \|_{W_\infty(\mD(A_\lambda),Y)} \leq M \| \mathbf{h} \|_V$ and $\| \tilde{u} - {u} \|_{L^2(0,\infty;U)} \leq M \| \mathbf{h} \|_V$, by the Lipschitz-continuity of the mappings $\mathcal{Y}$ and $\mathcal{U}$. It follows that the three quadratic terms in the above relation are of order $\| \mathbf{h} \|_V^2$ and thus that
\begin{align*}
& | \mathcal{V}(\tilde{\by}_0)- \mathcal{V}(\by_0)- \langle \bp(0), \tilde{\by}_0 - {\by}_0 \rangle_{V',V} | \\
& \qquad = \Big| \frac{1}{2} \| \tilde{\by}-\by \|_{Y}^2 +
\frac{\alpha}{2} \| \tilde{u}- {u} \|_{U}^2 -
\big\langle \bp, F(\tilde{\by}-\by) \big\rangle_{L^2(0,\infty;Y)} \Big| \leq M \| \mathbf{h} \|_V^2.
\end{align*}
This proves that $D\mathcal{V}(\by_0)= \bp(0),$ as announced.

Let $\by_0 \in B_V(\delta_5)$, set $(\by,u,\bp)= (\mathcal{Y}(\by_0),\mathcal{U}(\by_0),\mathcal{P}(\by_0))$ and choose $t \geq 0$. To verify \eqref{eq:costate_derValueFunc}, we define
\begin{equation*}
\tilde{\by}\colon s \geq 0 \mapsto \by(t+s), \quad
\tilde{u}\colon s \geq 0 \mapsto u(t+s), \quad
\tilde{\bp}\colon s \geq 0 \mapsto \bp(t+s).
\end{equation*}
By the dynamic programming principle, $\tilde{u}$ is the solution to problem \eqref{eq:NLQprob} with initial condition $\tilde{\by}(0)=\by(t)$. The associated trajectory and costate are $\tilde{\by}$ and $\tilde{\bp}$. Since $\| \by(t) \|_V \leq \delta_4$, we can use the above arguments to obtain that
$D\mathcal{V}(\tilde{\by}(0))= \tilde{\bp}(0)$
and finally that $D\mathcal{V}(\by(t))= \bp(t)$ in $V'$. Since $\bp\in L^2(0,\infty;Y)$, this equality also holds in $L^2(0,\infty;Y)$.
\end{proof}

\begin{proof}[Proof of Theorem \ref{theo19}]
 The main idea is to express the dynamics of the error $\be(t):=\bar{\by}(t)-\by_d(t)$ in feedback form by utilizing classical results on remainder terms for Taylor approximations. Let us detail the most important steps. First, for $\delta_6$ sufficiently small, from Corollary \ref{coro:diff} we know that $\mV$ is smooth and, hence, can be approximated by a Taylor series around $0$. From Theorem \ref{thm:subOptimality} it also follows that
\begin{align*}
| \mV(\by) - \sum_{k=2}^d \frac{1}{k!} \mT_k(\by,\dots,\by) | = o (\| \by \|_V^d), \quad \forall \by\in B_V(\delta_6).
\end{align*}
Since the $\mT_k$ are multilinear forms on $V^k$, the uniqueness of Taylor expansions implies that indeed $D^k\mV(0)= \mT_k$.
  Consequently, with \cite[Theorem 4A]{Zei86}, we also obtain a Taylor series expansion of $D\mV$ of the form
\begin{align*}
 D\mV(\by) = \sum_{k=2}^d \frac{1}{(k-1)!} \mT_k(\cdot,\by,\dots,\by) + R_d(\by), \quad \forall \by\in B_V(\delta_6)
\end{align*}
where the remainder term $R_d$ is given by
\begin{align*}
  R_d(\by)  =  \int_0^1 \frac{(1-\tau)^{d-1}}{(d-1)!} D^{d+1}\mV(\tau \by)(\cdot,\by,\dots,\by)\,\mathrm{d}\tau   .
\end{align*}
In particular, along the optimal trajectory $\bar{\by}(\cdot)=\mathcal{Y}(\by_0)(\cdot)$ it holds that
\begin{align}\label{eq:aux_taylor}
 \bp(t)=\mathcal{P}(\by_0)(t)=D\mV(\bar{\by}(t)) = \sum_{k=2}^d \frac{1}{(k-1)!} D^k \mV(0)(\cdot,\bar{\by}(t),\dots,\bar{\by}(t)) + R_d(\bar{\by}(t)).
\end{align}
By Proposition \ref{proposition:optiCondWeak} we know that $\bp \in L^2(0,\infty;Y)$. Similarly, since $\bar{\by}\in W_\infty(\mD(A),Y)$ from Theorem \ref{thm:existenceUniquenessLyapunov_2} it follows that
\begin{align*}
 \sum_{k=2}^{d} \frac{1}{(k-1)!} D^k \mV(0)(\cdot, \bar{\by}(\cdot),\dots, \bar{\by}(\cdot)) \in L^2(0,\infty;Y).
\end{align*}
Using \eqref{eq:aux_taylor} we then obtain $R_d(\bar{\by}(\cdot))\in L^2(0,\infty;Y)$. Consider now the error dynamics which satisfy
\begin{align*}
\dot{\be} = A_\pi \be - F(\bar{\by}) + F(\by_d)+ \sum_{k=3}^d (G_k(\bar{\by})- G_k(\by_d)) - \frac{1}{\alpha} BB^* R_d(\bar{\by}),
\end{align*}
with $G_k$ as in \eqref{eq:nonl_feed}. Defining a forcing term $\mathbf{f}$ via
\begin{align*}
 \mathbf{f}:=-F(\bar{\by})+F(\by_d)+ \sum_{k=3}^d (G_k(\bar{\by})-G_k(\by_d)) - \frac{1}{\alpha}BB^* R_d(\bar{\by})
\end{align*}
yields a system of the form $\dot{\be} = A_\pi \be +\mathbf{f}, \be(0)=0$. Moreover, for $\delta_6$ sufficiently small, we have the following estimate
\begin{align*}
 \| \mathbf{f} \|_{L^2(0,\infty;Y)} \le M (\tilde{\delta} \| \be \| _{W_\infty(\mD(A),Y)} + \| \by _0 \|_V^d),
\end{align*}
where the constant $M$ is independent of $\by_0$ and $\tilde{\delta}$ can be made arbitrarily small by reducing the value of ${\delta}_6$. This shows the first estimate. The estimate for the controls $\bar{u}$ and $u_d$ then follow exactly as in the proof of \cite[Theorem 22]{BreKP19}.
\end{proof}

\section{Linear optimality systems}\label{sec:lin_opt_sys}

Here, we analyze a class of linear optimality systems that arise in the proof of Lemma \ref{lemma:inverseMappingWeak}. With the space $X$ defined as in \eqref{eq:X_prod_space}, for a given $(\by_0,\bff,\bg,h)  \in X$, we consider:
\begin{equation*} \label{eq:LQprob} \tag{$LQ$}
\min_{\begin{subarray}{c} \by \in W_\infty(\mD(A_\lambda),Y) \\ u \in L^2(0,\infty;U) \end{subarray}}
 J[\bg,h](\by,u) \quad \text{subject to: }
e[\bff,\by_0](\by,u)= 0,
\end{equation*}
where $J	[\bg,h]\colon W_\infty(\mD(A_\lambda),Y) \times L^{2}(0,\infty;U) \to \mathbb R$ and $ e[\bff,\by_{0}]\colon W_\infty(\mD(A_\lambda),Y) \times L^{2}(0,\infty;U)\to L^{2}(0,\infty;Y)\times  Y$ are defined by
\begin{align*}
J[\bg,h](\by,u):= \ &
\frac{1}{2} \int_0^\infty \| \by \|_Y^2 \dd t
+ \langle \bg,\by \rangle_{L^2(0,\infty;[\mD(A_\lambda)]'),L^2(0,\infty;\mD(A_\lambda))}
+\frac{\alpha}{2} \int_0^\infty  \| u\|^2 _U \dd t + \langle h,u\rangle_{L^2(0,\infty;U)}  , \\
e[\bff,\by_0](\by,u):= \ & (\dot{\by} - (A\by + Bu + \bff),\by(0)-\by_0) \in L^2(0,\infty;Y)\times V.
\end{align*}

\begin{proposition} \label{prop:linOS}
For all $(\by_0,\bff,\bg,h) \in X$, there exists a unique triplet $(\by,u,\bp) \in W_\infty(\mD(A_\lambda),Y) \times L^2(0,\infty;U) \times W_\infty(Y,[\mD(A_\lambda)]')$ such that
\begin{equation} \label{eq:linOS}
\begin{cases}
\begin{array}{rll}
\dot{\by}-(A\by + Bu) = & \! \! \! \bff \qquad & \text{in $L^2(0,\infty;Y)$} \\
\by(0)=& \! \! \! \by_0 \qquad & \text{in $V$} \\
-\dot{\bp} - A' \bp - \by = & \! \! \! \bg & \text{in $L^2(0,\infty;[\mD(A_\lambda)]')$} \\
\alpha u + B^*  \bp = & \! \! \! - h & \text{in $L^2(0,\infty;U)$}.
\end{array}
\end{cases}
\end{equation}
Moreover there exists a constant $M>0$, independent of $(\bff,\bg,h,\by_0)$, such that
\begin{equation} \label{eq:linOSest}
\| (\by,u,\bp) \|_{W_\infty(\mD(A_\lambda),Y) \times L^2(0,\infty;U) \times W_\infty(Y,[\mD(A_\lambda)]')}
\leq M \| (\by_0,\bff,\bg,h) \|_X.
\end{equation}
\end{proposition}
\begin{proof}
For finite horizon problems the proof would be standard. For the infinite horizon case the result cannot readily be obtained from the literature, and thus we decided to provide a proof here.

 We prove the assertion with the help of the following two auxiliary statements.

 \emph{Claim 1.}  There exists a constant $M>0$ such that for all $(\bff,\bg,h,\by_0) \in X$, the linear-quadratic problem \eqref{eq:LQprob} has a unique solution $(\by,u)$ satisfying the following bounds:
\begin{equation} \label{eq:est_sol_LQ}
\| \by \|_{W_\infty(\mD(A_\lambda),Y)} \leq M \| (\by_0,\bff,\bg,h) \|_X \quad \text{and} \quad
\| u \|_{L^2(0,\infty;U)} \leq M \| (\by_0,\bff,\bg,h) \|_X.
\end{equation}

\emph{Proof of Claim 1.} Due to Consequence \ref{cons:C3}, problem \eqref{eq:LQprob} is feasible. Let us now consider a minimizing sequence $(\by_n,u_n)_{n \in \mathbb{N}}$. We can assume that for all $n \in \mathbb{N}$,
\begin{align} \label{eq:Existence1}
J[\bg,h](\by_n,u_n) \leq M \| (\by_0,\bff,\bg,h) \|_X^2.
\end{align}
Let us show that the sequence $(\by_n,u_n)$ is bounded in $W_\infty(\mD(A_\lambda),Y) \times L^2(0,\infty;U)$. By Young's inequality, for all $\varepsilon >0$ it holds that
\begin{align*}
J[\bg,h](\by_n,u_n)& \ge \frac{1}{2} \| \by_n \|_{L^2(0,\infty;Y)}^2 - \| \bg \|_{L^2(0,\infty;[\mD(A_\lambda)]')} \| \by _n \|_{L^2(0,\infty;\mD(A_\lambda))} \\
&\qquad  + \frac{\alpha}{2} \| u_n \|^2_{L^2(0,\infty;U)} - \| h \| _{L^2(0,\infty;U)} \| u_n \| _{L^2(0,\infty;U)} \\
& \ge \frac{1}{2} \| \by_n \|_{L^2(0,\infty;Y)}^2 - \frac{1}{2\varepsilon} \| \bg \|_{L^2(0,\infty;[\mD(A_\lambda)]')}^2 - \frac{\varepsilon}{2} \| \by _n \|_{L^2(0,\infty;\mD(A_\lambda))}^2 \\
&\qquad  + \frac{\alpha}{2} \left( \| u_n \|^2_{L^2(0,\infty;U)} -  \frac{\| h\| _{L^2(0,\infty;U)}}{\alpha} \right)^2 - \frac{\| h \| _{L^2(0,\infty;U)}^2}{2\alpha} .
\end{align*}
Combining this estimate with \eqref{eq:Existence1}, we obtain that
\begin{equation}\label{eq:yu_intermed}
\begin{aligned}
 &\max( \| \by_n \| _{L^2(0,\infty;Y)},\| u_n \|_{L^2(0,\infty;U)} ) \\
 &\qquad \qquad \le M \left( \| (\by_0,\bff,\bg,h) \|_X  +\sqrt{\varepsilon} \| \by_n \| _{L^2(0,\infty;\mD(A_\lambda))} + \frac{1}{\sqrt{\varepsilon}} \| \bg \| _{L^2(0,\infty;[\mD(A_\lambda)]')} \right).
\end{aligned}
\end{equation}
Consequence \ref{cons:C2} now implies that
\begin{align*}
\| \by _n \| _{W_\infty(\mD(A_\lambda),Y)} &\le  M \left( \| \by_0 \|_V + \| \bff \|_{L^2(0,\infty;Y)} + \| \by_n\|_{L^2(0,\infty;Y)} \right) \\[1ex]
& \le M \left( \| (\by_0,\bff,\bg,h) \|_X  +\sqrt{\varepsilon} \| \by_n \| _{L^2(0,\infty;\mD(A_\lambda))} + \frac{1}{\sqrt{\varepsilon}} \| \bg \| _{L^2(0,\infty;[\mD(A_\lambda)]')} \right),
\end{align*}
for some constant $M$ independent of $(\by_0,\bff,\bg,h)$ and $\varepsilon$. Choosing $\varepsilon$ sufficiently small, this yields
\begin{align*}
\| \by _n \| _{W_\infty(\mD(A_\lambda),Y)} &\le  M \| (\by_0,\bff,\bg,h) \|_X .
\end{align*}
Utilizing \eqref{eq:yu_intermed} we further arrive at
\begin{align*}
\| u_n \| _{L^2(0,\infty;U)} \le M \| (\by_0,\bff,\bg,h) \| _X.
\end{align*}
It follows that the sequence $(\by_n,u_n)$ is bounded in $W_\infty(\mD(A_\lambda),Y) \times L^2(0,\infty;U)$ and has a weak limit point $(\by,u)$ satisfying \eqref{eq:est_sol_LQ}. One can prove the optimality of $(\by,u)$ with the same techniques as those used for the proof of \cite[Proposition 2]{BreKP19}. The uniqueness of the solution directly follows from the linearity of the state equation and the strict convexity in $u$ of the cost functional.

\emph{Claim 2.} For all $(\by_0,\bff,\bg,h) \in X$, there exists a unique costate $\bp \in W_{\infty}(Y,[\mD(A_\lambda)]')$ satisfying the following relations:
\begin{align} \label{eq:costate_LQ}
-\dot{\bp} - A' \bp - \by = \ & \bg \\
\label{eq:control_LQ}
\alpha u + B^* \bp  = \ & -h.
\end{align}
Here $(\by,u)$ denotes the unique solution to \eqref{eq:LQprob}. Moreover, there exists a constant $M>0$ independent of $(\by_0,\bff,\bg,h)$ such that $\| \bp \|_{W_{\infty}(Y,[\mD(A_\lambda)]')} \leq M \| (\by_0,\bff,\bg,h) \|_X$.

\emph{Proof of Claim 2.} The mapping $e[\bff,\by_0]$ is continuous and affine, and thus its Fr\'{e}chet derivative is given by
\begin{align*}
 &De\colon W_\infty(\mD(A_\lambda),Y) \times L^2(0,\infty;U) \to L^2(0,\infty;Y)\times V\\
 &De(\bz,v)=(\dot{\bz}-(A\bz+Bv),\bz(0)).
\end{align*}
Let us argue that $De$ is surjective: Let $  (\mathbf{r},\mathbf{s})\in L^2(0,\infty;Y) \times V$ and consider
\begin{align*}
  \dot{\bz}-(A\bz+Bv)&=\mathbf{r}, \ \   \bz(0)= \mathbf{s}.
\end{align*}
By Consequence \ref{cons:C3} there exists $\bw \in W_\infty(\mD(A_\lambda),Y)$ such that
  \begin{align*}
    \dot{\bw} = (A-BK)\bw + \mathbf{r}, \  \bw(0) = \mathbf{s}.
  \end{align*}
Setting $v=-K\bw\in L^2(0,\infty;U)$ we have solved $De(\bz,v)=(\mathbf{r},\mathbf{s})$. The remaining arguments are similar to those provided in the proof of Proposition \ref{proposition:optiCondWeak} and are thus omitted here.

\end{proof}

\bibliographystyle{siam}
\bibliography{references}

\subsection*{Acknowledgement}

This work was partly supported by the ERC advanced grant 668998 (OCLOC) under
the EU's H2020 research program.
\end{document}